\documentclass[11pt]{article}
\usepackage{amsmath}
\usepackage[utf8]{inputenc}
\usepackage{tikz}
\usepackage{amsthm}
\usepackage{thmtools}
\usepackage{hyperref}
\usepackage{cleveref}
\usepackage[a4paper, total={7in, 9in}]{geometry}
\usepackage{blindtext}
\usepackage{comment}
\usepackage{geometry}
\providecommand{\keywords}[1]
{
	\small	
	\textbf{\textit{Keywords:}} #1
}
\usepackage{amsmath}
\NewDocumentEnvironment{alignb}{b}{%
	\begin{align*}
		\refstepcounter{equation} #1 \tag{\theequation}
	\end{align*}
}{}
\allowdisplaybreaks
\usepackage{enumerate}
\makeatletter
\newcommand{\myitem}[1]{%
	\item[#1]\protected@edef\@currentlabel{#1}%
}
\makeatother
\usepackage{clipboard}
\declaretheorem[numberwithin=section]{theorem, definition}
\declaretheorem{lemma, proposition, remark}[style=plain,
numberwithin=section]
\numberwithin{equation}{section}
\newtheorem{mytheorem}{Theorem}
\newtheorem{defi}{Definition}[section]
\numberwithin{mytheorem}{section}
\usepackage{amsbsy}
\usepackage{graphicx}
\usepackage[T1]{fontenc}
\usepackage{lmodern}
\usepackage{imakeidx}
\usepackage{amssymb}
\usepackage{thmtools}
\usepackage{cite}
\newenvironment{myproof}[2] {\paragraph{Proof of {#1} {#2} :}}{\hfill$\square$}

\usepackage[english]{babel}
\title{Existence of multiple normalized solutions to a critical growth Choquard equation involving mixed operator}
\author{Nidhi Nidhi\footnote{Department of Mathematics, Indian Institute of Technology, Delhi, Hauz Khas, New Delhi-110016, India. e-mail: nidhi.nidhi@maths.iitd.ac.in}\; and K. Sreenadh\footnote{Department of Mathematics, Indian Institute of Technology, Delhi, Hauz Khas, New Delhi-110016, India. e-mail: sreenadh@maths.iitd.ac.in}\;}
\date{}
\begin{document}
	\maketitle
	\begin{abstract}
		\noindent In this paper we study the normalized solutions of the following critical growth Choquard equation with mixed local and non-local operators: 
		\begin{equation*}
			\begin{array}{rcl}
				-\Delta u +(-\Delta)^s u & = & \lambda u +\mu |u|^{p-2}u +(I_{\alpha}*|u|^{2^*_{\alpha}})|u|^{2^*_{\alpha}-2}u \text{ in } \mathbb{R}^N\\
				\left\| u \right\|_2 & =  & \tau,
			\end{array}
		\end{equation*}
		here $N\geq 3$, $\tau>0$, $I_{\alpha}$ is the Riesz potential of order $\alpha\in (0,N)$, $2^*_{\alpha}=\frac{N+\alpha}{N-2}$ is the critical exponent corresponding to the Hardy Littlewood Sobolev inequality, $(-\Delta)^s$ is the non-local fractional Laplacian operator with $s\in (0,1)$, $\mu>0$ is a parameter and $\lambda$ appears as Lagrange multiplier. We have shown the existence of atleast two distinct solutions in the presence of mass subcritical perturbation, $\mu |u|^{p-2}u$ with $2<p<2+\frac{4s}{N}$ under some assumptions on $\tau$.\\
		\noindent \keywords{ Normalized solution, Choquard equation, critical exponent, mixed local and non-local operator, $L^2$-subcritical perturbation, nonlinear Scr$\ddot{\text{o}}$dinger equation driven by mixed operator.}
	\end{abstract}
	\section{Introduction}
	This article concerns the existence of multiple normalized solutions to the following critical growth Choquard equation involving mixed diffusion-type operator:
	\begin{equation}\label{prob}
		\begin{array}{rcl}
			-\Delta u +(-\Delta)^s u & = & \lambda u +\mu |u|^{p-2}u +(I_{\alpha}*|u|^{2^*_{\alpha}})|u|^{2^*_{\alpha}-2}u \text{ in } \mathbb{R}^N\\
			\left\| u \right\|_2 & =  & \tau,
		\end{array}
	\end{equation}
	where $N\geq 3$, $\tau>0$, $2<p<2+\frac{4s}{N}$, $\mu>0$ is a parameter and $\lambda$ appears as Lagrange multiplier. The fractional Laplace operator $(-\Delta)^s$ is defined as follows:
	\begin{equation*}
		(-\Delta)^s u = \frac{C(N,s)}{2}\text{P.V} \int_{\mathbb{R}^N}\frac{u(x)-u(y)}{|x-y|^{N+2s}}dy,
	\end{equation*}
	with P.V being the abbreviation for principal value, and $C(N,s)$ is a normalizing constant, refer \cite{Nezza2012hitchhiker} for a clearer understanding.
	For the sake of convenience, we will take $C(N,s)=2$. Here,	$I_{\alpha}$ is the Riesz potential of order $\alpha\in (0,N)$ given by
	\begin{equation}\label{A_alpha}
		I_{\alpha}(x)=\frac{A_{N,\alpha}}{|x|^{N-\alpha}}\text{ with } A_{N,\alpha}=\frac{\Gamma(\frac{N-2}{2})}{\pi^{\frac{N}{2}}2^{\alpha}\Gamma(\frac{\alpha}{2})}\text{ for every }x\in\mathbb{R}^N\setminus \{0\},
	\end{equation}
	and $2^*_{\alpha}=\frac{N+\alpha}{N-2}$, is the critical exponent with respect to the following well known Hardy-Littlewood-Sobolev(HLS) inequality:
	\begin{proposition}\label{prop1.1}
		Let $t,r>1$ and $0<\alpha <N$ with $1/t+1/r=1+\alpha/N$, $f\in L^t(\mathbb{R}^N)$ and $h\in L^r(\mathbb{R}^N)$. There exists a sharp constant $C(t,r,\alpha,N)$ independent of $f$ and $h$, such that
		\begin{equation}\label{HLS}
			\int_{\mathbb{R}^N}\int_{\mathbb{R}^N}\frac{f(x)h(y)}{|x-y|^{N-\alpha}}~dxdy \leq C(t,r,\alpha,N) \|f\|_{L^t}\|h\|_{L^r}.
		\end{equation}
		If $t=r=2N/(N+\alpha)$, then
		\begin{align}\label{C_alpha}
			C(t,r,\alpha,N)=C(N,\alpha)= \pi^{\frac{N-\alpha}{2}}\frac{\Gamma(\frac{\alpha}{2})}{\Gamma(\frac{N+\alpha}{2})}\left\lbrace \frac{\Gamma(\frac{N}{2})}{\Gamma(N)}\right\rbrace^{-\frac{\alpha}{N}}.
		\end{align}
		Equality holds in  \eqref{HLS} if and only if $\frac{f}{h}\equiv constant$ and
		$\displaystyle h(x)= A(\gamma^2+|x-a|^2)^{(N+\alpha)/2}$
		for some $A\in \mathbb{C}, 0\neq \gamma \in \mathbb{R}$ and $a \in \mathbb{R}^N$.
	\end{proposition}
	\noindent From this inequality, it follows that 
	\begin{align*}
		{\mathcal{A}_q(u):=	\int_{\mathbb R^N}\int_{\mathbb R^N}\frac{|u(x)|^{q}|u(y)|^{q}}{|x-y|^{N-\alpha}}~dxdy}
	\end{align*}
	is well defined if $\frac{N+\alpha}{N}\leq q \leq \frac{N+\alpha}{N-2}=2^*_{\alpha} $. The exponent $q = 2^*_{\alpha}$ is known as Hardy-Littlewood-Sobolev critical exponent and similar to the usual critical exponent, $H^1_0(\Omega)\ni\, u\mapsto \mathcal{A}_{2^*_\alpha}(u)$ is continuous for the norm topology but not for the weak topology (see \cite{Moroz2017guide}). Thus, the presence of this HLS critical exponent ($2^*_{\alpha}$) makes our problem challenging and intriguing to work on. Equations involving nonlinearity of the form $(I_{\alpha}*|u|^q)|u|^{q-2}u$ are called {\it Choquard equation}, as in 1976, Choquard, at the Symposium on Coulomb Systems  
	utilised the energy functional associated to equation 
	\begin{equation}
		\left\{ \begin{array}{rl}   	
			& 	-\Delta u  +  u = (I_2*|u|^2)u\;\;\text{in } \mathbb{R}^3,\\
			&  	u \in H^1(\mathbb{R}^3),
		\end{array}
		\right.
	\end{equation}
	to examine a viable approximation to Hartree-Fock theory for a one-component plasma (see \cite{lieb1977existence}). The equation has various other applications in quantum physics, for instance, it is used to characterise an electron confined within its own vacancy, see \cite{penrose1996gravity} and related sources. Several works have ever since conducted research on the existence, multiplicity, and qualitative characteristics of the solution to the problem 
	\begin{equation}\label{Norm_Choq}
		\begin{array}{rl}   	
			& 	-\Delta u  +  \lambda u =\mu (I_\alpha*|u|^p)|u|^{p-2}u\;\;\text{in } \mathbb{R}^N,
			%				&  	\int_{\mathbb{R}^N}|u|^2dx  =  \tau.
		\end{array}
	\end{equation}
	as detailed in \cite{filippucci2020singular, Moroz2013groundstates, liu2022another}. We are interested in discussing the multiplicity of normalized solutions to a critical growth Choquard equation involving mixed local $(\Delta)$ and non-local operator $(-\Delta)^s$. 
	
	The mixed operator $\mathcal{L}=-\Delta+(-\Delta)^s$, generally comes into the picture, whenever the impact on a physical phenomenon is due to both local and non-local changes. Some of its applications can be seen in bi-model power law distribution processes (see \cite{pagnini2021should}). A variety of contributions have examined issues related to the existence of solutions, their regularity and symmetry properties, Neumann problems, Green's function estimates and eigen values (see, for example, \cite{biagi2021global,biagi2022mixed,abatangelo2021elliptic,arora2021combined,Divya2019Eigenvalue}).

	% has many applications, such as bi-model power law distribution processes (see \cite{pagnini2021should}). Such an operator comes into the picture, whenever the impact on a physical phenomenon is due to both local and non-local changes. A variety of contributions have examined issues related to the existence of solutions, their regularity and symmetry properties, Neumann problems, and Green's function estimates (see, for example, \cite{biagi2021global,biagi2022mixed,abatangelo2021elliptic}).  
	% For additional investigation on the semilinear elliptic equations involving operator $\mathcal{L}$, we quote \cite{arora2021combined} and references therein. 
	%A more general mixed operator has also been investigated by certain researchers. As an illustration, the authors of \cite{Divya2019Eigenvalue} looked at the second eigen value for the mixed operator that employs both usual $p$-laplacian $(-\Delta_p)$ and non-local p-laplacian $(-\Delta_{J,p})$.

	The study of \eqref{prob} has physical relevance, as it provides us the standing wave solution for the nonlinear Schr$\ddot{\text{o}}$dinger (NLS) equation driven by mixed local and nonlocal operators given as follows:
	\begin{equation}\label{Schrodinger_equation}
		i \frac{\partial \psi}{\partial t} = -\Delta \psi+(-\Delta)^s\psi -\mu |\psi|^{p-2} \psi -(I_{\alpha}*|\psi|^{2^*_{\alpha}})|\psi|^{2^*_{\alpha}-2}\psi.
	\end{equation}
	A standing wave solution is of the form $\psi(x,t)=e^{-i\lambda t} u(x)$, where $\lambda\in \mathbb{R}$ and $u\in H^1(\mathbb{R}^N)$ solves:
	\begin{equation}\label{1.6}
		-\Delta u +(-\Delta)^s u  =  \lambda u +\mu |u|^{p-2}u +(I_{\alpha}*|u|^{2^*_{\alpha}})|u|^{2^*_{\alpha}-2}u \text{ in } \mathbb{R}^N.
	\end{equation}
	The additional $L^2-$norm constraint in \eqref{prob} gives us a standing wave with prescribed mass.
	While addressing solutions to \eqref{1.6}, there exists two schools of thought. The initial approach involves fixing a $\lambda\in \mathbb{R}$ and thereafter looking for the critical points of the associated energy functional, whereas the other method, that we are following here, is to fix the $L^2$-norm, that is, to search for the critical points of
	$$E(u):=\frac{\left\| \nabla u \right\|_2^2}{2}+\frac{[u]^2}{2}-\mu \frac{\left\| u \right\|_p^p}{p}-\frac{A(u)}{22^*_{\alpha}}; \text{ where } [u]^2=\int_{\mathbb{R}^N}\int_{\mathbb{R}^N}\frac{|u(x)-u(y)|^2}{|x-y|^{N+2s}}dxdy,$$
	%	 \text{ where }[u]^2 \text{ and } A(u)\text{ will be formally defined}=\int_{\mathbb{R}^N}\int_{\mathbb{R}^N}\frac{|u(x)-u(y)|^2}{|x-y|^{N+2s}}dxdy,$$
	restricted to the manifold $S(\tau):=\{u\in H^1(\mathbb{R}^N): \left\| u \right\|_2=\tau\}$, here $A(u)=\mathcal{A}_{2^*_{\alpha}}(u)$. The previous method has already been extensively employed, however the latter one is new and appears more captivating, in this case $\lambda$ playing the role of the Lagrange multiplier is also a part of the unknown and the solution thus found is called {\it normalized solution}. Recently, the study of normalized solutions has attracted the researchers, formally, the solution of the following constrained problem is called the normalized solution
	\begin{equation}\label{Norm_sol}
		\left\{ \begin{array}{rl}   	
			& 	-\Delta u  =  \lambda u +g(u)\;\;\text{in } \mathbb{R}^N,\\
			&  	\int_{\mathbb{R}^N}|u|^2dx  =  c.
		\end{array}
		\right.
	\end{equation} 
	%	 The reason for studying the equation \eqref{Norm_sol} is the physical incentive it provides, as its solution yields stationary states of a nonlinear Schr$\ddot{\text{o}}$dinger equation with a predetermined $L^2-$norm.  
	Jeanjean in \cite{jeanjean1997existence}, demonstrated the existence of a radial solution for equation \eqref{Norm_sol} subject to certain assumptions on the function $g$. 
	Further, the existence of infinitely many solutions to \eqref{Norm_sol} with $c=1$ under same assumptions on $g$ has been shown by Bartsch and De Valeriola in \cite{bartsch2012normalized}.
	In \cite{noris2015existence}, Noris et. al. explored the normalized solutions in the context of bounded domains with Dirichlet boundary conditions. Normalized solutions have been seen to exist for $p$ values within the intervals $(1,1+\frac{4}{N})$, $(1+\frac{4}{N}, 2^*-1)$, and $p=1+\frac{4}{N}$, under certain requirements on $c$, with the domain being unit ball and $g(t)=|t|^{p-1}t$. Furthermore, the authors in \cite{pierotti2017normalized} have tackled the issue in general bounded domains.  The existence of normalized solutions of nonlinear Schrödinger systems has been extensively explored. Interested readers can refer to the references \cite{gou2018multiple, bartsch2016normalized, bartsch2018normalized, bartsch2019multiple, noris2014stable, noris2019normalized}. The study of quadratic ergodic mean field games system also investigates normalized solutions type, as discussed in \cite{pellacci2021normalized}.
	
	Let us formally initiate our study by discussing the variational framework of the problem \eqref{prob}.
	\begin{defi}\label{weak_sol}
		A function $u\in S(\tau)$ is said to be a solution to \eqref{prob} if it satisfies the following:
		\begin{equation}\label{sol}
			\int_{\mathbb{R}^N}\nabla u \nabla v +\ll u,v \gg = \lambda\int_{\mathbb{R}^N}uv+ \mu\int_{\mathbb{R}^N}|u|^{p-2}uv +\int_{\mathbb{R}^N}(I_{\alpha}*|u|^{2^*_{\alpha}})|u|^{2^*_{\alpha}-2}uv, 
		\end{equation}
		for all $v\in H^1(\mathbb{R}^N)$. Here $$\ll u, v \gg := \int_{\mathbb{R}^N}\int_{\mathbb{R}^N}\frac{(u(x)-u(y))(v(x)-v(y))}{|x-y|^{N+2s}}dxdy,$$
		and the space $H^1(\mathbb{R}^N)$ is equipped with the norm:
		$$\left\| u \right\| = \left(T(u)^2+\left\| u \right\|_2^2\right)^{\frac{1}{2}} \text{ where }T(u)^2=\left\| \nabla u \right\|_2^2+[u]^2.$$
	\end{defi}
	\noindent Using the Pohozaev identity, it is seen that a solution to \eqref{prob} lies on the Pohozaev Manifold 
	$$\mathcal{M}_{\tau}:=\{u\in S(\tau): M(u)=0\},$$ $$\text{ where }M(u)= \left\| \nabla u \right\|_2^2+s[u]^2-\mu\gamma_p\left\| u \right\|_p^p-A(u) \text{ with }\gamma_p:=\frac{N(p-2)}{2p}$$
	further using the fibre maps technique in section 2, we subdivided $\mathcal{M}_{\tau}$ into disjoint subsets $\mathcal{M}_{\tau}^+$ and $\mathcal{M}_{\tau}^-$. The idea is to look for distinct solutions in these disjoint subsets. \\
	Let $S$ be the  best constant corresponding to the imbedding $D^{1,2}(\mathbb{R}^N)\hookrightarrow L^{2^*}(\mathbb{R}^N)$. By \cite{shang2023normalized}, we know that
	\begin{equation}\label{S_alpha}
		S_{\alpha}=\inf_{u\in D^{1,2}(\mathbb{R}^N)\setminus \{0\}}\frac{\left\| \nabla u\right\|_2^2}{A(u)^{\frac{1}{2^*_{\alpha}}}}=\frac{S}{(A_{\alpha}C_{\alpha})^{\frac{1}{2^*_{\alpha}}}},
	\end{equation}
	and $S_{\alpha}$ is achieved by the family of functions of the form:
	\begin{equation}\label{U_epsilon}
		U_{\epsilon,x_0}(x)=\frac{(N(N-2)\epsilon^2)^{\frac{N-2}{4}}}{(\epsilon^2+|x-x_0|^2)^{\frac{N-2}{2}}},\text{ for } x_0\in \mathbb{R}^N \text{ and } \epsilon>0,
	\end{equation}
	here $A_{\alpha}=A_{N,\alpha}$ and $C_{\alpha}=C({N,\alpha})$ given in \eqref{A_alpha} and \eqref{C_alpha} respectively. Thanks to symmetric decreasing rearrangement, the Gagliardo-Nirenberg inequality (see \cite[Theorem 1.1]{Fiorenza2021detailed}) precisely, 
	\begin{equation}\label{G_N_inequality}
		\left\| u \right\|_\beta\leq C_{N,\beta}\left\| \nabla u \right\|_2^{\theta}\left\| u \right\|_2^{1-\theta} \text{ where } \theta=\frac{N(\beta-2)}{2\beta} \text{ for all }\beta\in [2,2^*],
	\end{equation}
	and compact imbedding $H_r(\mathbb{R}^N)\hookrightarrow L^q(\mathbb{R}^N)$ for all $q\in (2,2^*)$ (\cite[Lemma~3.1.4]{Badiale2010Semilinear}), by the Ekeland variational principle, we could deduce the existence of first solution. Taking 
	$$\tau_0=\left(\frac{p(2^*_{\alpha}-1)}{\mu C_{N,p}(22^*_{\alpha}-p\gamma_p)}\left(\frac{(2-p\gamma_p)2^*_{\alpha}S_{\alpha}^{2^*_{\alpha}}}{22^*_{\alpha}-p\gamma_p}\right)^{\frac{2-p\gamma_p}{2(2^*_{\alpha}-1)}}\right)^{\frac{1}{p(1-\gamma_p)}},$$
	and $$\tau_1=\left(\frac{2(2^*_{\alpha}-1)}{\gamma_p^{\frac{p\gamma_p}{2}}\mu C_{N,p}(22^*_{\alpha}-p\gamma_p)}\left(\frac{pS_{\alpha}^{\frac{2^*_{\alpha}}{2^*_{\alpha}-1}}}{2-p\gamma_p}\right)^{\frac{2-p\gamma_p}{2}}\right)^{\frac{1}{p(1-\gamma_p)}},$$
	we have the following:
	\begin{mytheorem}\label{Theorem 1}
		For $N\geq 3$, $s\in (0,1)$, $2<p<2+\frac{4s}{N}$ and $0<\tau<\min\{\tau_0,\tau_1\}$, there exists a radially symmetric function $u_{\tau}^+\in H^1(\mathbb{R}^N)$ that attains $m_{\tau}^+:=\displaystyle \inf_{u\in \mathcal{M}_{\tau}^+}E(u)$, that is, $E(u_{\tau}^+)=m_{\tau}^+<0$. Moreover, $u_{\tau}^+$ solves \eqref{prob} corresponding to some $\lambda_{\tau}^+<0$, for sufficiently large $\mu>0$.
	\end{mytheorem}
	%   \noindent Here, $$\tau_0=\left(\frac{p(2^*_{\alpha}-1)}{\mu C_{N,p}(22^*_{\alpha}-p\gamma_p)}\left(\frac{(2-p\gamma_p)2^*_{\alpha}S_{\alpha}^{2^*_{\alpha}}}{22^*_{\alpha}-p\gamma_p}\right)^{\frac{2-p\gamma_p}{2(2^*_{\alpha}-1)}}\right)^{\frac{1}{p(1-\gamma_p)}},$$
	% and $$\tau_1=\left(\frac{2(2^*_{\alpha}-1)}{\gamma_p^{\frac{p\gamma_p}{2}}\mu C_{N,p}(22^*_{\alpha}-p\gamma_p)}\left(\frac{pS_{\alpha}^{\frac{2^*_{\alpha}}{2^*_{\alpha}-1}}}{2-p\gamma_p}\right)^{\frac{2-p\gamma_p}{2}}\right)^{\frac{1}{p(1-\gamma_p)}}.$$
	\noindent Since our problem involves mass subcritical perturbation, $2<q<2+\frac{4s}{N}$, \cite{Jeanjean2022multiple} motivates us to expect a second solution.
	Denoting $m_{\tau}^-=\displaystyle \inf_{u\in \mathcal{M}_{\tau}^-}E(u)$, 
	% using $\epsilon$-blow up analysis, 
	in section 4 we deduced a relation between $m_{\tau}^+$ and $m_{\tau}^-$, that helped us to prove the existence of the second solution to \eqref{prob}. Precisely, we have the following result:
	\begin{mytheorem}\label{Theorem 2}
		Let $N\geq 3$, $2<p<2+\frac{4s}{N}$, $0<\tau<\min\{\tau_0,\tau_1\}$ and $\mu>0$ be sufficiently large, then $m_{\tau}^-$ is achieved by a radially symmetric function $u_{\tau}^-\in H^1(\mathbb{R}^N)$. Furthermore, $u_{\tau}^-$ solves \eqref{prob} corresponding to some $\lambda_{\tau}^-<0$.
	\end{mytheorem}
	%	\subsection*{Notations}
	%	\begin{itemize}
		%	\item $$E(u)=\frac{\left\| \nabla u \right\|_2^2}{2} +\frac{[u]^2}{2}-\mu \frac{\left\| u \right\|_p^p}{p}-\frac{A(u)}{2.2^*_{\alpha}}$$
		%		\item $M(u)= \left\| u \right\|_2^2+s[u]^2-\mu\gamma_p\left\| u \right\|_p^p-A(u)$, where $\gamma_p=\frac{N(p-2)}{2p}$.
		%	\item $T(u)= \left(\left\| \nabla u \right\|_2^2+[u]^2\right)^{\frac{1}{2}} $
		%	\item  $S$ is the best constant corresponding to the imbedding $D^{1,2}(\mathbb{R}^N)\hookrightarrow L^{2^*}(\mathbb{R}^N)$. By \cite{shang2023normalized}, we know that
		%	\begin{equation}\label{S_alpha}
			%	S_{\alpha}=\inf_{u\in D^{1,2}(\mathbb{R}^N)\setminus \{0\}}\frac{\left\| \nabla u\right\|_2^2}{A(u)^{\frac{1}{2^*_{\alpha}}}}=\frac{S}{(A_{\alpha}C_{\alpha})^{\frac{1}{2^*_{\alpha}}}},
			%	\end{equation}
		%		and $S_{\alpha}$ is achieved by the family of functions of the form:
		%		\begin{equation}\label{U_epsilon}
			%		U_{\epsilon,x_0}(x)=\frac{(N(N-2)\epsilon^2)^{\frac{N-2}{4}}}{(\epsilon^2+|x-x_0|^2)^{\frac{N-2}{2}}},\text{ for } x_0\in \mathbb{R}^N \text{ and } \epsilon>0,
			%	\end{equation}
		%	here $A_{\alpha}=A_{N,\alpha}$ and $C_{\alpha}=C({N,\alpha})$ given in \eqref{A_alpha} and \eqref{C_alpha} respectively.
		%\end{itemize}
		\section{Preliminaries}
		\noindent In this section, we will establish the necessary groundwork required to deduce the final existence results.
		\begin{lemma}\label{Lemma 2.1}
			If $u\in S(\tau)$ is a solution of \eqref{prob}, corresponding to some $\lambda \in \mathbb{R}$, then 
			$u \in \mathcal{M}_{\tau}.$
		\end{lemma}
		\begin{proof}
			Since, $u\in S(\tau)$ solves \eqref{prob}, for some $\lambda\in \mathbb{R}$, then we have:
			\begin{equation}\label{2.1}
				\lambda \left\| u \right\|_2^2 = \left\| u \right\|_2^2+[u]^2-\mu\left\| u \right\|_p^p-A(u),
			\end{equation}
			also, 
			%by \cite{Avenia2015} and \cite{Moroz2015}, it can be seen that 
			$u$ satisfies the following Pohozaev identity:
			\begin{equation}\label{2.2}
				\left(\frac{N-2}{2}\right)\left\| \nabla u \right\|_2^2 +\left(\frac{N-2s}{2}\right) [u]^2= \frac{N\lambda}{2}\left\| u \right\|_2^2+\frac{N}{p}\mu\left\| u \right\|_p^p+\left(\frac{N+\alpha}{22^*_{\alpha}}\right)A(u),
			\end{equation}
			see \cite[Theorem~A1]{Giacomoni2025Normalized} and \cite[Theorem~2.5]{Anthal2025Pohozaev}.
			Using \eqref{2.1} in \eqref{2.2}, we get
			$$M(u)= \left\| \nabla u \right\|_2^2+s[u]^2-\mu\gamma_p\left\| u \right\|_p^p-A(u)=0,$$
			where $\gamma_p= \frac{N(p-2)}{2p}$.
		\end{proof}
		\noindent This Pohozaev manifold $\mathcal{M}_{\tau}$, will be playing a crucial role in the study of existence and multiplicity results. We will further subdivide it into following three disjoint subsets:
		$$\mathcal{M}_{\tau}^0:= \{u\in \mathcal{M}_{\tau} : 2 \left\| \nabla u \right\|_2^2+2s^2[u]^2=p\gamma_p^2\mu\left\| u \right\|_p^p+2.2^*_{\alpha}A(u)\},$$
		$$\mathcal{M}_{\tau}^+:= \{u\in \mathcal{M}_{\tau} : 2 \left\| \nabla u \right\|_2^2+2s^2[u]^2>p\gamma_p^2\mu\left\| u \right\|_p^p+2.2^*_{\alpha}A(u)\},$$			$$\mathcal{M}_{\tau}^-:= \{u\in \mathcal{M}_{\tau} : 2 \left\| \nabla u \right\|_2^2+2s^2[u]^2<p\gamma_p^2\mu\left\| u \right\|_p^p+2.2^*_{\alpha}A(u)\},$$
		and deduce the existence of a solution in $\mathcal{M}_{\tau}^+$ and another one in $\mathcal{M}_{\tau}^-$. As we move forward, it will become clearer why $\mathcal{M}_{\tau}^+$, $\mathcal{M}_{\tau}^-$, and $\mathcal{M}_{\tau}^0$ were chosen in this way. Now, for any $u\in S(\tau)$, by \eqref{S_alpha} and Gagliardo-Nirenberg inequality \eqref{G_N_inequality}
		%(see \cite[Theorem 1.1]{Fiorenza2021detailed}), precisely, $\left\| u \right\|_\beta\leq C_{N,\beta}\left\| \nabla u \right\|_2^{\theta}\left\| u \right\|_2^{1-\theta}$ where $\theta=\frac{N(\beta-2)}{2\beta}$,
		%and \eqref{S_alpha} 
		we have:
		\begin{equation}\label{2.3}
			E(u)  =  \frac{T(u)^2}{2}-\mu\frac{\left\| u \right\|_p^p}{p} -\frac{A(u)}{22^*_{\alpha}}
			\geq  \frac{T(u)^2}{2}-\frac{\mu C_{N,p}}{p}T(u)^{p\gamma_p}\tau^{p-p\gamma_p}-\frac{T(u)^{22^*_{\alpha}}}{22^*_{\alpha}S_{\alpha}^{2^*_{\alpha}}}.
		\end{equation}
		Defining,
		$$h(t):=\frac{t^2}{2}-\frac{\mu C_{N,p}t^{p\gamma_p}\tau^{p-p\gamma_p}}{p}-\frac{t^{22^*_{\alpha}}}{22^*_{\alpha}S_{\alpha}^{2^*_{\alpha}}} \text{ for all } t>0,$$
		we get $E(u)\geq h(T(u))$. Let us discuss some properties of the function $h$, that will be helpful for us.
		\begin{lemma}\label{Lemma 2.2}
			There exists $\tau_0>0$, such that
			for $\tau<\tau_0$, $h$ has a strict local minimum at negative level, a global maximum at  positive level and, we can find two positive constants $R_1>R_0$ such that $h(R_0)=0=h(R_1)$ with $h(t)>0$ if and only if $t\in (R_0,R_1)$.
		\end{lemma}
		\begin{proof}
			Define $$\bar{h}(t):= \frac{t^{2-p\gamma_p}}{2}-\frac{\mu C_{N,p}}{p}\tau^{p(1-\gamma_p)}-\frac{t^{22^*_{\alpha}-p\gamma_p}}{22^*_{\alpha}S_{\alpha}^{2^*_{\alpha}}} \text{ for } t>0,$$
			then $h(t)=t^{p\gamma_p}\bar{h}(t)$, and hence $h(t)>0$ if and only if $\bar{h}(t)>0$. Clearly, since $\bar{h}$ has unique critical point $t_0=\left( \frac{(2-p\gamma_p)2^*_{\alpha}S_{\alpha}^{2^*_{\alpha}}}{22^*_{\alpha}-p\gamma_p}\right)^{\frac{1}{2(2^*_{\alpha}-1)}}$, $\bar{h}$ is increasing in $(0,t_0)$, decreasing in $(t_0,\infty)$, $\bar{h}(0)=-\frac{\mu C_{N,p}}{p}\tau^{p(1-\gamma_p)}$, and $\bar{h}(t_0)>0 \text{ for all } \tau<\tau_0,$
			%=\left(\frac{p(2^*_{\alpha}-1)}{\mu C_{N,p}(22^*_{\alpha}-p\gamma_p)}\left(\frac{(2-p\gamma_p)2^*_{\alpha}S_{\alpha}^{2^*_{\alpha}}}{22^*_{\alpha}-p\gamma_p}\right)^{\frac{2-p\gamma_p}{2(2^*_{\alpha}-1)}}\right)^{\frac{1}{p(1-\gamma_p)}},$$ 
			it's curvature can be visualised as follows:
			
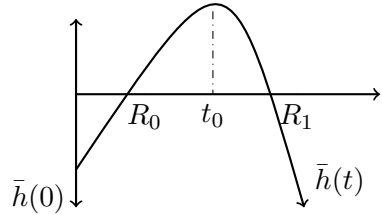
\begin{center}
				\begin{tikzpicture}
					\draw[thick,->] (0,0)--(4,0);
					\draw[thick,<->](0,1)--(0,-1.5);
					\draw[thick,->](0,-1).. controls (2,2) .. (3,-1.5) node[anchor= south west]{$\bar{h}(t)$};
					\draw (2.9,0)--(2.9,0) node[anchor=north]{$R_1$};
					\draw (0.9,0)--(0.9,0) node[anchor=north]{$R_0$};
					\draw [dash dot] (1.8,1.1)--(1.8,0) node[anchor=north]{$t_0$};
					\draw (0,-1)--(0,-1) node[anchor=north east]{$\bar{h}(0)$};
				\end{tikzpicture}
			\end{center}
			%	\begin{center}
				%	{\large\begin{tikzpicture}
						%	\draw[thick,->] (0,0)--(6,0) node[anchor=north west]{$t$};
						%	\draw[thick,<->] (0,3)--(0,-3);
						%		\draw (0.8,0)--(0.8,0) node[anchor = north]{$R_0$};
						%		\draw (3.8,0)--(3.8,0) node[anchor=north]{$R_1$};
						%		\draw[dash dot, thick] (2,2)--(2,0) node[anchor= north]{$t_0$};
						%		\draw[thick,->] (2,2) parabola (4.1,-2.41) node[anchor= south west]{$\bar{h}(t)$};
						%		\draw[thick] (2,2) parabola (0,-2) node[anchor=north east]{$\bar{h}(0)$};
						%		\end{tikzpicture}}
				%	\end{center}
			Thus, there exists $0<R_0<R_1$ such that $h(R_0)=0=h(R_1)$ and $h(t)>0$ if and only if $t\in (R_0,R_1)$. 
			Next we claim that $h$ has exactly two non-zero critical points.
			Now, since 
			$$h'(t)=t^{p\gamma_p-1}\left(t^{2-p\gamma_p}-\mu\gamma_pC_{N,p}\tau^{p(1-\gamma_p)}-\frac{t^{22^*_{\alpha}-p\gamma_p}}{S_{\alpha}^{2^*_{\alpha}}}\right),$$
			if $h$ has more than two non-zero critical points, then the function $g$, defined as
			$$g(t):=t^{2-p\gamma_p}-\frac{t^{22^*_{\alpha}-p\gamma_p}}{S_{\alpha}^{2^*_{\alpha}}},$$
			attains $C_{\tau}=\mu \gamma_pC_{N,p}\tau^{p(1-\gamma_p)}$ atleast thrice and hence, has at least two critical points. But, since $\bar{t}=\left(\frac{(2-p\gamma_p)S_{\alpha}^{2^*_{\alpha}}}{22^*_{\alpha}-p\gamma_p}\right)^{\frac{1}{2(2^*_{\alpha}-1)}}$ is the unique critical point of $g$, we get a contradiction. Thus, $h$ has atmost two non-zero crtical points. Also, since $h(t)\rightarrow 0^-$ as $t\rightarrow 0^+$ and $h(t)\rightarrow -\infty$ as $t\rightarrow \infty$, $h$ can exhibit the following geometry:
			
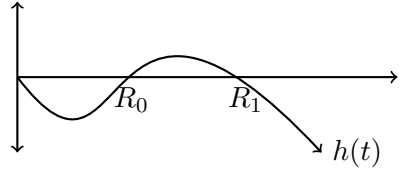
\begin{center}
				\begin{tikzpicture}
					\draw[thick,->](0,0)--(5,0);
					\draw[thick,<->] (0,1)--(0,-1);
					\draw (1.5,0)--(1.5,0) node[anchor=north]{$R_0$};
					\draw (3,0)--(3,0) node[anchor=north]{$R_1$};
					\draw[thick,->](0,0).. controls (1.5,-2) and (1,2.25) .. (4,-1) node[anchor= west]{$h(t)$};
				\end{tikzpicture}
			\end{center}
			Hence, we are done.
		\end{proof}
		\noindent For any $u\in H^1(\mathbb{R}^N)$, let us define the fiber maps $\star$ and $\circledast$, as follows:
		$$(t\star u)(x):=e^{\frac{Nt}{2}}u(e^tx)\text{ for } t\in \mathbb{R}; \text{ and } (t\circledast u)(x):=t^{\frac{N}{2}}u(tx) \text{ for } t\geq 0.$$
		Clearly, $e^t\circledast u= t\star u$. Now, defining $\psi_{u}(t):=E(t\star u)$, one can notice that $M(t\star u )=\psi_{u}'(t)$, also we have the following results about $\psi_{u}$.
		\begin{lemma}\label{Lemma 2.3}
			Let $u\in S(\tau)$ and $\tau <\tau_0$, then $\psi_{u}$ has exactly two zeroes and two critical points, that is, we can find unique $a_u<b_u<c_u<d_u$ such that $\psi_{u}'(a_u)=0=\psi_{u}'(c_u)$ and $\psi_{u}(b_u)=0=\psi_{u}(d_u)$. Moreover, we have the following:
			\begin{enumerate}
				\item $a_u\star u \in \mathcal{M}_{\tau}^+$ and $c_u\star u\in \mathcal{M}_{\tau}^-$. If $t\star u \in \mathcal{M}_{\tau}$, then either $t=a_u$ or $t=c_u$ and hence $\mathcal{M}_{\tau}^0$ is empty.
				\item $E(c_u\star u)=\max\{E(t\star u): t\in \mathbb{R}\}>0$ and $\psi_{u}$ is strictly decreasing in $(c_u,\infty)$.
				\item $T(t\star u)\leq R_0$ for every $t<b_u$ and $$E(a_u \star u)=\min\{E(t\star u): t\in \mathbb{R} \text{ and } T(t\star u)\leq R_0\}<0.$$
				\item The maps $\Phi_1:\mathcal{M}_{\tau}\rightarrow \mathbb{R}$ and $\Phi_2:\mathcal{M}_{\tau}\rightarrow \mathbb{R}$ defined as $\Phi_1(u):=a_u$ and $\Phi_2(u):=c_u$, are of class $C^1$.
			\end{enumerate}
		\end{lemma}
		\begin{proof}
			Since,		
			$$\psi_{u}(t)= E(t\star u)= \frac{e^{2t}}{2}\left\| \nabla u \right\|_2^2+\frac{e^{2st}}{2}[u]^2-\frac{\mu e^{p\gamma_pt}}{p}\left\| u \right\|_p^p-\frac{e^{22^*_{\alpha}t}}{22^*_{\alpha}}A(u),$$
			we get $$\psi_{u}'(t)= e^{22^*_{\alpha}t}\left(e^{(2-22^*_{\alpha})t}\left\| \nabla u \right\|_2^2+se^{(2s-22^*_{\alpha})t}[u]^2-\gamma_p\mu e^{(p\gamma_p-22^*_{\alpha})t}\left\| u \right\|_p^p-A(u)\right).$$
			If $\psi_{u}$ has more than two critical points, then the function $g$ defined as:
			$$g(t):=e^{(2-22^*_{\alpha})t}\left\| \nabla u \right\|_2^2+se^{(2s-22^*_{\alpha})t}[u]^2-\gamma_p\mu e^{(p\gamma_p-22^*_{\alpha})t}\left\| u \right\|_p^p,$$
			attains $A(u)$ atleast thrice and hence has atleast two critical points. Now, since	$$g'(t)= e^{(p\gamma_p-22^*_{\alpha})t}(\bar{g}(t)-C_p)$$
			where $ \bar{g}(t)= (2-22^*_{\alpha})e^{(2-p\gamma_p)t}\left\| \nabla u \right\|_2^2+s(2s-22^*_{\alpha})e^{(2s-p\gamma_p)t}[u]^2$ and $C_p=\mu\gamma_p(p\gamma_p-22^*_{\alpha})\left\| u \right\|_p^p$, $\bar{g}$ must attain $C_p$ atleast twice and hence have atleast one critical point.  But, 
			$$\bar{g}'(t)= (2-2.2^*_{\alpha})(2-p\gamma_p)e^{(2-p\gamma_p)t}\left\| \nabla u \right\|_2^2+s(2s-22^*_{\alpha})(2s-p\gamma_p)e^{(2s-p\gamma_p)t}[u]^2>0,$$
			for all $ t\in \mathbb{R}$, thus we get contradiction. Hence $\psi_{u}$ has atmost two crtical points. Further, since $t\mapsto T(t\star u)$ is continuous and increasing map from $\mathbb{R}$ onto $(0,+\infty)$, we can find $t_1,t_2\in \mathbb{R}$ such that $R_0=T(t_1\star u)<T(t\star u)<T(t_2\star u)=R_1$ for all $t\in (t_1,t_2)$, by \eqref{2.3} and \autoref{Lemma 2.2}
			$$\psi_{u}(t)=E(t\star u)\geq h(T(t\star u))>0 \text{ for all } t\in (t_1,t_2).$$
			Also, one can see that $\psi_{u}(t)\rightarrow -\infty$ as $t\rightarrow +\infty$ and $\psi_{u}(t)\rightarrow 0^-$ as $t\rightarrow -\infty$, because $p\gamma_p<2s<2<22^*_{\alpha}$. Thus, $\psi_{u}$ can have the following curvature:
			
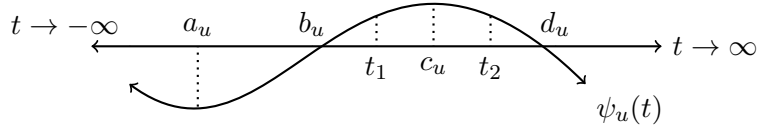
\begin{center}
				\begin{tikzpicture}\label{fig 1}
					\draw[thick,<->] (-4.5,0)--(3,0) node[anchor=west]{$t\rightarrow \infty$};
					\draw (-4,0)--(-4,0)node[anchor= south east]{$t\rightarrow -\infty$};
					\draw[thick,dotted] (-0.75,0.4)--(-0.75,0) node[anchor=north]{$t_1$};
					\draw [thick,dotted](0,0.5)--(0,0)node[anchor=north]{$c_u$};
					\draw[thick,dotted] (0.75,0.4)--(0.75,0) node[anchor=north]{$t_2$};
					\draw [thick,<->] (-4,-0.5)..controls (-2, -2) and (-1,2.5)..(2,-0.5) node[anchor=north west]{$\psi_u(t)$};
					\draw [thick,dotted](-3.1,-0.8)--(-3.1,0) node[anchor=south]{$a_u$};
					\draw (-1.6,0)--(-1.6,0) node[anchor=south]{$b_u$};
					\draw (1.6,0)--(1.6,0) node[anchor=south]{$d_u$};
				\end{tikzpicture}
			\end{center}
			therefore, $\psi_u$ has exactly two critical points, corresponding to a local minima $(a_u)$ at negative level and global maxima $(c_u)$ at positive level, and exactly two roots ($b_u$ and $d_u$).
			\begin{enumerate}
				\item  Since, $a_u$ is a strict local minima of $\psi_{u}$, 
				$M(a_u\star u) = \psi_{u}'(a_u)=0$,  and 
				\begin{eqnarray*}
					0  < \psi_{u}''(a_u)& = & 2e^{2a_u}\left\| \nabla u \right\|_2^2+2s^2e^{2sa_u}[u]^2-\mu p\gamma_p^2e^{p\gamma_pa_u}\left\| u \right\|_p^p-22^*_{\alpha}e^{22^*_{\alpha}}A(u)\\
					& = & 2\left\| \nabla a_u\star u \right\|_2^2+2s^2[a_u \star u]^2-\mu p\gamma_p^2\left\| a_u\star u \right\|_p^p-22^*_{\alpha}A(a_u\star u),
				\end{eqnarray*}		 	 	
				thus $a_u \star u \in \mathcal{M}_{\tau}^+$. Similarly, since $c_u$ is global maxima of $\psi_u$, we will get $c_u\star u \in \mathcal{M}_{\tau}^-$. 
				Now, if $t\star u \in \mathcal{M}_{\tau}$, then clearly $t$ is a critical point of $\psi_u$, hence either $t=a_u$ or $t=c_u$. Moreover, since $\psi_{u}$ has exactly two crtical points, both corresponding to its extremas, $\mathcal{M}_{\tau}^0$ must be an empty set.
				\item It is evident by the curvature of $\psi_{u}$.
				\item By monotonicity of the surjective map $t\mapsto T(t\star u)$ onto $(0,\infty)$, it is clear that $T(t\star u)\leq T(t_1\star u)=R_0$ for all $t<b_u\leq t_1$. Moreover, since $\psi_u$ is decreasing in $(-\infty, a_u)$ and increasing in $(a_u,t_1]$,
				\begin{equation*}
					0>E(a_u\star u)= \psi_u(a_u)  =  \min\{\psi_u(t): t\leq t_1\} = \min\{E(t\star u): T(t\star u)\leq T(t_1\star u)=R_0\}.
				\end{equation*}
				\item By implicit function theorem, as done in the proof of Lemma 3.3 in \cite{Han2022Normalized}, clearly $\Phi_1$ and $\Phi_2$ are of class $C^1$.
			\end{enumerate}
		\end{proof}
		\begin{lemma}\label{Lemma 4.3}
			If $u\in \mathcal{M}_{\tau}$ is a critical point of $E|_{\mathcal{M}_{\tau}}$, then $u$ is a critical point of $E|_{S(\tau)}$.
		\end{lemma}
		\begin{proof}
			For a critical point $u $ of $E|_{\mathcal{M}_{\tau}}$, by the Lagrange multiplier's rule, there exists $\lambda_1$ and $\lambda_2\in \mathbb{R}$ such that:
			$$E'(u)(v)-\lambda_1 \int_{\mathbb{R}^N}uv-\lambda_2 M'(u)(v)=0 \text{ for all } v\in H^1(\mathbb{R}^N),$$
			that is,
			\begin{eqnarray*}
				(1-2\lambda_2)\int_{\mathbb{R}^N}\nabla u \nabla v +(1-2\lambda_2 s)\ll u,v \gg & = & \mu(1-\lambda_2p\gamma_p)\int_{\mathbb{R}^N}|u|^{p-2}uv+\lambda_1 \int_{\mathbb{R}^N}uv\\
				&& +(1-\lambda_2 22^*_{\alpha})\int_{\mathbb{R}^N}(I_{\alpha}*|u|^{2^*_{\alpha}})|u|^{2^*_{\alpha}-2}uv,
			\end{eqnarray*}
			for all $v\in H^1(\mathbb{R}^N)$, and hence, $u$ solves:
			\begin{equation}\label{4.19}
				-(1-2\lambda_2)\Delta u +(1-2\lambda_2 s)(-\Delta)^su =\lambda_1 u +\mu(1-\lambda_2 p\gamma_p )|u|^{p-2}u+(1-\lambda_222^*_{\alpha})(I_{\alpha}*|u|^{2^*_{\alpha}})|u|^{2^*_{\alpha}-2}u,
			\end{equation}
			in $\mathbb{R}^N$.
			Claim: $\lambda_2=0$. \\
			Now, as done in the proof of \autoref{Lemma 2.1}, by \eqref{4.19} we have:
			\begin{equation*}
				(1-2\lambda_2)\left\| \nabla u \right\|_2^2+(1-2s\lambda_2)[u]^2 = \lambda_1 \left\| u \right\|_2^2+\mu(1-\lambda_2 p\gamma_p)\left\| u \right\|_p^p+(1-\lambda_2 22*_{\alpha})A(u),
			\end{equation*} 
			and
			\begin{eqnarray*}
				\lambda_1 \left\| u \right\|_2^2 & = & \frac{2}{N}\left((1-2\lambda_2)\left(\frac{N-2}{2}\right)\left\| \nabla u \right\|_2^2+(1-2s\lambda_2)\left(\frac{N-2s}{2}\right)[u]^2\right.\\
				&&-\mu(1-\lambda_2 p\gamma_p)\frac{N}{p}\left\| u \right\|_p^p
				\left. - (1-\lambda_2 22^*_{\alpha})\left(\frac{N+\alpha}{22^*_{\alpha}}\right)A(u)\right),
			\end{eqnarray*}
			thus $$\lambda_2\left(2 \left\| \nabla u \right\|_2^2+2s^2[u]^2-\mu p\gamma_p\left\| u \right\|_p^p-22^*_{\alpha}A(u)\right)=0.$$
			Since, $\mathcal{M}_{\tau}^0$ is empty set, we must have $\lambda_2=0$. Therefore, $u$ is a critical point of $E|_{S(\tau)}$.
		\end{proof}
		\noindent 
		%Let us define $$m_{\tau}^-:=\inf_{u\in \mathcal{M}_{\tau}^-}E(u);\text{ and } m_{\tau}^{+}:=\inf_{u\in \mathcal{M}_{\tau}^+}E(u),$$
		%further
		For any $k>0$, denoting $A_k=\{u\in S(\tau): T(u)<k\}$, we define 
		$$m_{\tau}:=\displaystyle \inf_{u\in A_{R_0}}E(u),$$ 
		where $R_0$ is as deduced in \autoref{Lemma 2.2},
		then we have the following results for $m_{\tau}$, $m_{\tau}^-$ and $m_{\tau}^+$:
		\begin{lemma}\label{Lemma 2.4}
			$m_{\tau}^->0.$
		\end{lemma}
		\begin{proof}
			For any $u\in \mathcal{M}_{\tau}^-$ we have, $0\star u= u\in \mathcal{M}_{\tau}^-$, then by \autoref{Lemma 2.3},  $0$ is the global maxima of $\psi_u$ at a positive level and	$E(u)=\psi_{u}(0)=\max\{E(t\star u): t\in \mathbb{R}\}>0$, hence $m_{\tau}^-\geq 0$. Moreover, for every $u\in \mathcal{M}_{\tau}^-$, we can find some $t_u\in \mathbb{R}$ such that $T(t_u\star u)=t_0$, where $t_0$ is the global maxima of $h$ deduced in \autoref{Lemma 2.2}. Thus,
			$$E(u)=\max\{E(t\star u): t\in \mathbb{R}\}\geq E(t_u\star u)\geq h(T(t_u\star u))=h(t_0)>0 \text{ for all } u\in \mathcal{M}_{\tau}^-,$$
			hence $m_{\tau}^-\geq h(t_0)>0$.
		\end{proof}
		\begin{lemma}\label{Lemma 2.5}
			$\displaystyle \sup_{u\in \mathcal{M}_{\tau}^+}E(u)\leq 0 < m_{\tau}^-$ and $\mathcal{M}_{\tau}^+\subset A_{R_0}$.
		\end{lemma}
		\begin{proof}
			Clearly, for any $u\in \mathcal{M}_{\tau}^+$, $a_u=0$, thus by \autoref{Lemma 2.3} $E(u)<0$ and hence by \autoref{Lemma 2.4} $\displaystyle \sup_{u\in \mathcal{M}_{\tau}^+}E(u)\leq 0 <m_{\tau}^-$. Further, $T(u)=T(a_u\star u)<T(t_1 \star u)=R_0$, for all $u\in \mathcal{M}_{\tau}^+$, since $0=a_u<t_1$. Hence $\mathcal{M}_{\tau}^+\subset A_{R_0}$.
		\end{proof}
		\begin{lemma}\label{Lemma 2.6}$\displaystyle -\infty <m_{\tau}=\inf_{u\in \mathcal{M}_{\tau}}E(u)=m_{\tau}^+<0,$ and
			for $\delta>0$ small enough
			\begin{equation}\label{2.4}
				m_{\tau}<\displaystyle \inf_{\bar{A}_{R_0}\setminus A_{R_0-\delta}}E(u).
			\end{equation}
		\end{lemma}
		\begin{proof}
			For any $u\in A_{R_0}$, we have:
			$$E(u)\geq h(T(u))\geq \min_{t\in [0,R_0]}h(t)>-\infty,$$
			and hence $m_{\tau}>-\infty$. Also, since $a_u\star u \in \mathcal{M}_{\tau}^+\subset A_{R_0}$, $$-\infty< m_{\tau}=\displaystyle \inf_{u\in A_{R_0}}E(u)\leq E(a_u\star u)=\psi_u(a_u)<0.$$ 
			Further, if $u\in A_{R_0}$, then by \autoref{Lemma 2.3} $E(u)=E(0\star u)\geq E(a_u\star u)\geq m_{\tau}^+$, hence $m_{\tau}\geq m_{\tau}^+$. Also since $\mathcal{M}_{\tau}^+\subset A_{R_0}$ we get $m_{\tau}=m_{\tau}^+$. Now, since $\mathcal{M}_{\tau}=\mathcal{M}_{\tau}^+\cup \mathcal{M}_{\tau}^+\cup\mathcal{M}_{\tau}^0$, $\mathcal{M}_{\tau}^0$ is an empty set and 
			$$m_{\tau}^+=\inf_{u\in \mathcal{M}_{\tau}^+}E(u)\leq \sup_{u\in \mathcal{M}_{\tau}^+}E(u)\leq 0<\inf_{u\in \mathcal{M}_{\tau}^-}E(u),$$
			by \autoref{Lemma 2.5}, then clearly $\displaystyle \inf_{u\in \mathcal{M}_{\tau}}E(u)=\inf_{u\in \mathcal{M}_{\tau}^+}E(u)=m_{\tau}^+$. Therefore,
			$$-\infty<m_{\tau}=\inf_{u\in \mathcal{M}_{\tau}}E(u)=m_{\tau}^+<0.$$
			Now, since $h$ is continuous, $h(R_{0})=0$, $h(t)<0$ for all $t\in (0,R_0)$ and $m_{\tau}<0$, we can find $\delta>0$ small enough so that $h(t)\geq \frac{m_{\tau}}{2}$ for all $t\in [R_0-\delta,R_0]$. Hence, for all $u\in \bar{A}_{R_0}\setminus A_{R_0-\delta}$,
			$$R_0-\delta<T(u)\leq R_0 \Rightarrow E(u)\geq h(T(u))\geq \frac{m_{\tau}}{2}>m_{\tau}.$$
			Thus, we get \eqref{2.4}.
		\end{proof}
		\section{First solution}
		In this section, using the above prerequisite results, symmetric decreasing rearrangement, and Ekeland variational principle, we will be showing the existence of a radially symmetric function $u_{\tau}^+\in \mathcal{M}_{\tau}^+$ and $\lambda_{\tau}^+<0$, such that $(u_{\tau}^+,\lambda_{\tau}^+)$ solves \eqref{prob}. The subsequent rearrangement inequalities will be beneficial for this purpose.
		\begin{remark}\label{Remark 3.1}
			For any $u\in H^1(\mathbb{R}^N)$, let $u^*$ be its symmetric decreasing rearrangement, then we have the following:
			\begin{enumerate}
				\item $\left\| u \right\|_q= \left\| u^*\right\|_q$ for all $q\in [2,2^*]$,
				\item $A(u)\leq A(u^*)$,
				\item $\left\| \nabla u^* \right\|_2\leq \left\| \nabla u \right\|_2$ and $[u^*]^2\leq [u]^2$.
			\end{enumerate}
		\end{remark}		
		\noindent Interested readers can go through \cite{Baernstein2019Symmetrization, Burchard2009Short, Lieb2001Analysis} and \cite[Remark~2.1]{Nidhi2025Normalized} to see the proof.
		%Taking $$\tau_1=\left(\frac{2(2^*_{\alpha}-1)}{\gamma_p^{\frac{p\gamma_p}{2}}\mu C_{N,p}(22^*_{\alpha}-p\gamma_p)}\left(\frac{pS_{\alpha}^{\frac{2^*_{\alpha}}{2^*_{\alpha}-1}}}{2-p\gamma_p}\right)^{\frac{2-p\gamma_p}{2}}\right)^{\frac{1}{p(1-\gamma_p)}},$$ 
		%We prove \autoref{Theorem 1} as follows:
		%	\begin{theorem}\label{Theorem 1}
			%	For $N\geq 3$, $s\in (0,1)$, $2<p<2+\frac{4s}{N}$ and $0<\tau<\min\{\tau_0,\tau_1\}$, there exists a radially symmetric function $u_{\tau}^+\in H^1(\mathbb{R}^N)$ that attains $m_{\tau}$, that is, $E(u_{\tau}^+)=m_{\tau}=m_{\tau}^+<0$. Moreover, $u_{\tau}^+$ solves \eqref{prob} corresponding to some $\lambda_{\tau}^+<0$, for sufficiently large $\mu>0$.
			%	\end{theorem}
		\begin{myproof}{Theorem}{\ref{Theorem 1}}
			Let $\{w_n\}\subset A_{R_0}$ be the minimizing sequence for $E$ on $A_{R_0}$, then taking $w_n^*$ to be the symmetric decreasing rearrangement of $w_n$. By the rearrangement inequalities, \autoref{Remark 3.1}, it can be seen that $\{w_n^*\}\subset A_{R_0}$ and $E(w_n^*)\leq E(w_n)$ for each $n\in \mathbb{N}$, thus, $\{w_n^*\}$ is a minimizing sequence as well. Now, for each $n\in \mathbb{N}$, by \autoref{Lemma 2.3} there exists $a_{n}\in \mathbb{R}$ such that $a_n\star w_n^*\in \mathcal{M}_{\tau}^+$ and $E(w_n^*)=E(0\star w_n^*)\geq E(a_n\star w_n^*)$.
			Taking $v_n=a_n\star w_n^*$ to be the minimizing sequence for $E$ on $\mathcal{M}_{\tau}^+$ and hence, that of $E$ on $A_{R_0}$, clearly, $v_n$ is radially symmetric and $T(v_n)<R_0-\delta$ for all $n\in \mathbb{N}$.
			%and
			Applying Ekeland variational principle 
			%on the functional $I(u)=E(a_u\star u)$ defined on $S(\tau)$
			(see Theorem 1.1 and its corollaries in \cite{Ghoussoub}) we can find a sequence of radially symmetric functions, $\{u_n\}$
			%=\{\bar{a}_n\star \bar{u}_n\}$ in $\mathcal{M}_{\tau}^+$ 
			such that\begin{equation}\label{3.1}
				\left\{
				\begin{array}{cc}
					E(u_n)\rightarrow m_{\tau} & \text{ as } n\rightarrow \infty,\\
					E(u_n)\leq E(v_n) & \text{ for all } n\in \mathbb{N},\\
					M(u_n) \rightarrow 0 & \text{ as } n\rightarrow \infty,\\
					%					\left\| u_n-v_n \right\|_{H^1(\mathbb{R}^N)} \rightarrow 0 & \text{ as } n\rightarrow \infty,\\
					E'_{S(\tau)}(u_n) \rightarrow 0 & \text{ as } n\rightarrow \infty.\\
				\end{array}
				\right.
			\end{equation}
			Here, $E'_{S(\tau)}(u_n) \rightarrow 0$ means that the sequence  $y_n=\sup\left\{\frac{E'(u_n)(w)}{\left\| w \right\|}: w\in S({\tau})\right\}$ converges to $0$. Now, by \eqref{3.1} and the method of Lagrange multipliers, we can find a sequence $\{\lambda_n\}$ such that:
			\begin{equation}\label{3.2}
				E'(u_n)-\lambda_n\Phi'(u_n) \rightarrow 0, \text{ where } \Phi(u)=\frac{1}{2}\left\| u \right\|_2^2.
			\end{equation}
			%	Claim: $\{\bar{\lambda}_n\}\rightarrow 0$.\\
			%	Since $\{J_n'(u_n)\}\rightarrow 0$, where $J_n(u)=E(u)-\lambda_n\Phi(u)-\bar{\lambda}_nM(u)$, we have:
			%	\begin{eqnarray*}
				%  0 & = & \lim_{n\rightarrow \infty}\frac{d}{dt}(J_n(t\star u_n))_{\{t=0\}} = \lim_{n\rightarrow \infty} \frac{d}{dt}\left(E(t\star u_n)-\lambda_n\Phi(t\star u_n)-\bar{\lambda}_nM(t\star u_n)\right)_{\{t=0\}}\\
				% & = & \lim_{n\rightarrow \infty}\left(\psi_{u_n}'(0)-\bar{\lambda}_n\psi_{u_n}''(0)\right)=\lim_{n\rightarrow \infty}(-\psi_{u_n}''(0))\bar{\lambda}_n
				%	\end{eqnarray*}
			Clearly, since $\{u_n\}\subset A_{R_0}$, it is bounded in $H^1(\mathbb{R}^N)$ and hence, weakly convergent upto a subsequence in $H^1(\mathbb{R}^N)$. Denoting the subsequence by $\{u_n\}$ itself, let $u_0\in H^1(\mathbb{R}^N)$ be such that $u_n\rightharpoonup u_0$. Clearly, $u_0\in H_r(\mathbb{R}^N)$.\\
			Claim 1: $\lambda_n \rightarrow \lambda<0$, up to a subsequence.\\
			Clearly,
			\begin{equation}\label{new_eq}
				o_n(1)= \left\| \nabla u_n \right\|_2^2+[u_n]^2-\mu \left\| u_n \right\|_p^p-A(u_n)-\lambda_n\tau^2,
			\end{equation}
			by weak convergence of $\{u_n\}$ and \eqref{3.2}.
			%\begin{equation}\label{new_eq}
			%	\left\| \nabla u_n \right\|_2^2+[u_n]^2-\mu \left\| u_n \right\|_p^p-A(u_n)-\lambda_n\tau^2=o_n(1).
			%\end{equation}
			%$\displaystyle \lim_{n\rightarrow \infty} \left(E'(u_n)(u_n)-\lambda_n \Phi'(u_n)(u_n)\right)=0$, thus
			%\begin{equation}\label{3.3}
			%	\left\| \nabla u_n \right\|_2^2+[u_n]^2-\mu\left\| u_n\right\|_p^p-A(u_n)-\lambda_n\tau^2=o_n(1),
			%\end{equation}
			%\begin{equation}\label{3.3} 
			%	0=\lim_{n\rightarrow \infty} E'(u_n)(u_n)-\lambda_n \Phi'(u_n)(u_n)= \lim_{n\rightarrow \infty}\left(\left\| \nabla u_n \right\|_2^2+[u_n]^2-\mu\left\| u_n\right\|_p^p-A(u_n)-\lambda_n\tau^2\right),
			%\end{equation}
			%$$0=\lim_{n\rightarrow \infty} E'(v_n)(v_n)-\lambda_n \Phi'(u_n)(u_n)= \lim_{n\rightarrow \infty}\left(\left\| \nabla u_n \right\|_2^2+[u_n]^2-\left\| u_n\right\|_p^p-A(v_n)-\lambda_n\tau\right),$$
			%and, since $\{u_n\}$ is bounded,  
			Then, by Fatou's lemma and the compact imbedding of $H_r(\mathbb{R}^N)$ in $L^p(\mathbb{R}^N)$ (see \cite[Lemma~3.1.4]{Badiale2010Semilinear}) we have:
			$$\lambda_n \leq\frac{T(u_n)^2}{\tau^2}-\frac{\mu \left\| u_0\right\|_p^p}{\tau^2}-\frac{A(u_0)}{\tau^2}+o(1),$$
			hence by boundedness of $\{u_n\}$ in $H^1(\mathbb{R}^N)$, 
			$$|\tau^2\lambda_n|\leq |T(u_n)^2|+\mu \left\| u_0 \right\|_p^p+|A(u_0)|+o(1)<+\infty.$$
			%$$|\lambda_n|\leq \frac{\left(\left\| \nabla u_n \right\|_2^2+[u_n]^2-\mu\left\| u_n\right\|_p^p-A(u_n)\right)}{\tau^2}\leq \frac{\left\| u_n \right\|_{H_r(\mathbb{R}^N)}}{\tau^2}<+\infty \text{ for large } n\in \mathbb{N}.$$
			Thus $\{\lambda_n\}$ is bounded and hence convergent upto a subsequence. Denoting the subsequence by $\{\lambda_n\}$ itself, let $\lambda_0\in \mathbb{R}$ be such that $\lambda_n\rightarrow \lambda_0$. Now, since $u_n\in \mathcal{M}_{\tau}$, by \eqref{new_eq} and the fact that $\gamma_p<1$ we get:
			\begin{eqnarray*}
				\lambda_0\tau^2 & = & \lim_{n\rightarrow \infty}\left( \left\| \nabla u_n \right\|_2^2+[u_n]^2-\mu\left\| u_n\right\|_p^p-A(u_n)\right)\\
				& = & \lim_{n\rightarrow \infty} \left((1-s)[u_n]^2+\mu\left(\gamma_p-1\right)\left\| u_n \right\|_p^p\right)< 0,
			\end{eqnarray*}
			for sufficiently large $\mu>0$.\\
			Claim 2: $u_0\neq 0$.\\
			Suppose $u_0=0$, then by the compact embedding $H_r(\mathbb{R}^N)\hookrightarrow L^q(\mathbb{R}^N)$ for all $q\in (2,2^*)$ and \eqref{3.1}, we get $\displaystyle \lim_{n\rightarrow \infty}A(u_n)=\lim_{n\rightarrow \infty}T_s(u_n)^2$, where $T_s(u):=(\left\| \nabla u \right\|_2^2+s[u]^2)^{\frac{1}{2}}$. Suppose $T_s(u_n)^2\rightarrow l$, then by \eqref{S_alpha}
			$$l\leq \frac{l^{2^*_{\alpha}}}{S_{\alpha}^{2^*_{\alpha}}}\Rightarrow l(S_{\alpha}^{2^*_{\alpha}}-l^{2^*_{\alpha}-1})\leq 0.$$
			%thus either $l=0$ or $l\geq S_{\alpha}^{\frac{N+\alpha}{\alpha+2}}$.
			Since $m_{\tau}<0$, $l=0$ will lead us to a contradiction, because if $l=0$, then
			\begin{equation*}
				m_{\tau}=\lim_{n\rightarrow \infty} E(u_n)\geq \lim_{n\rightarrow \infty}\left(\frac{T_s(u_n)^2}{2}-\frac{\mu \left\| u_n \right\|_p^p}{p}-\frac{A(u_n)}{22^*_{\alpha}}\right)=0.
			\end{equation*}
			Hence we must have $l\geq S_{\alpha}^{\frac{N+\alpha}{\alpha+2}}$. Now,
			\begin{eqnarray*}
				m_{\tau} & = & \lim_{n\rightarrow \infty}E(u_n) =\lim_{n\rightarrow \infty}\left(E(u_n)-\frac{M(u_n)}{22^*_{\alpha}}\right)\\
				& = & \lim_{n\rightarrow \infty} \left(\left(\frac{2^*_{\alpha}-1}{22^*_{\alpha}}\right)\left\| \nabla u_n \right\|_2^2+\left(\frac{2^*_{\alpha}-s}{22^*_{\alpha}}\right)[u_n]^2+\left(\frac{1}{p}-\frac{\gamma_p}{22^*_{\alpha}}\right)\left\| u_n \right\|_p^p\right)\\
				& \geq & \left(\frac{2^*_{\alpha}-1}{22^*_{\alpha}}\right)\lim_{n\rightarrow \infty}T_s(u_n)^2= \left(\frac{2^*_{\alpha}-1}{22^*_{\alpha}}\right)l\geq \left(\frac{2^*_{\alpha}-1}{22^*_{\alpha}}\right)S_{\alpha}^{\frac{N+\alpha}{\alpha+2}}\geq 0,
			\end{eqnarray*}
			thus, we are again lead to a contradiction. Therefore $u_0\neq 0$.\\
			Claim 3: $(u_0, \lambda_0)$ solves \eqref{prob}.\\
			Since $\lambda_0<0$, we can define the following equivalent norm on $H^1(\mathbb{R}^N)$:
			
			$$\left\| u \right\|_{\lambda_0}:=(\left\|\nabla u \right\|_2^2+[u]^2-\lambda_0\left\| u \right\|_2^2)^{\frac{1}{2}}.$$
			%forms an equivalent norm on $H^1(\mathbb{R}^N)$, 
			Then for any $v\in H^1(\mathbb{R}^N)$, by \eqref{3.2} we have:
			\begin{eqnarray}\label{3.4}
				0 & = & \lim_{n\rightarrow \infty}\left(E'(u_n)(v)-\lambda_n\Phi'(u_n)(v)\right)\nonumber\\
				& = & \int_{\mathbb R^N}\nabla u_0\nabla v +\ll u_0, v \gg- \lambda_0\int_{\mathbb{R}^N}u_0v-A'(u_0)(v)-\mu\int_{\mathbb{R}^N}|u_0|^{p-2}u_0 v,
			\end{eqnarray}
			since the mappings, $u\mapsto \frac{\left\|u \right\|_p^p}{p}$ and $A$ defined on $H^1(\mathbb{R}^N)$ are of class $C^1$. Thus, $u_0$ solves:
			$$-\Delta u_0 +(-\Delta)^su_0 = \lambda_0u_0+\mu|u_0|^{p-2}u_0+(I_{\alpha}*|u_0|^{2^*_{\alpha}})|u|^{2^*_{\alpha}-2}u_0 \text{ in } \mathbb{R}^N.$$
			Next, we will show that $\left\| u_0 \right\|_2=\tau$.
			Following the proof of \autoref{Lemma 2.1}, we have $M(u_0)=0$. Now, define $\bar{u}_n:=u_n-u_0$.
			%we will see that this $\{\bar{u}_n\}$ converges strongly to zero, for $\tau<\tau_1$. 
			Since $\bar{u}_n\rightharpoonup 0$ in $H^1(\mathbb{R}^N)$ and hence in $H_r(\mathbb{R}^N)$, then by Brezis Lieb lemma, lemma 2.4 of \cite{Moroz2013groundstates} and compact imbedding of $H_r(\mathbb{R}^N)$ in $L^p(\mathbb{R}^N)$, we get
			\begin{equation}\label{3.5}
				\begin{array}{rcl}
					\left\| \nabla \bar{u}_n \right\|_2^2 & = & \left\| \nabla u_n \right\|_2^2-\left\| u_0 \right\|_2^2+o_n(1)\\
					\left[\bar{u}_n\right]^2 & = & [u_n]^2-[u_0]^2+o_n(1)\\
					A(\bar{u}_n) & = & A(u_n)-A(u_0)+o_n(1),\\
					\left\| \bar{u}_n \right\|_p^p & = & o_n(1).
				\end{array}
			\end{equation}
			Now, by \eqref{3.5},
			\begin{eqnarray*}
				\lim_{n\rightarrow \infty}M(\bar{u}_n) & = & \lim_{n\rightarrow \infty}\left(\left\| \nabla \bar{u}_n \right\|_2^2+s[\bar{u}_n]^2-\mu \gamma_p\left\| \bar{u}_n\right\|_p^p-A(\bar{u}_n)\right)\\
				& = & \lim_{n\rightarrow \infty}\left(\left\| \nabla u_n \right\|_2^2+s[u_n]^2-A(u_n)- (\left\| \nabla u_0 \right\|_2^2+s[u_0]^2-A(u_0))\right) \\
				& = & \lim_{n\rightarrow \infty} \left(M(u_n)-\mu\gamma_p\left\| u_n \right\|_p^p-M(u_0)+\mu\gamma_p\left\| u_0 \right\|_p^p \right)= 0.
			\end{eqnarray*}
			Therefore, $\displaystyle \lim_{n\rightarrow \infty}\left(\left\|\nabla \bar{u}_n\right\|_2^2+s[\bar{u}_n]^2\right)=\lim_{n\rightarrow \infty}\left(\mu\gamma_p\left\| \bar{u}_n\right\|_p^p+A(\bar{u}_n)\right)=\lim_{n\rightarrow \infty}A(\bar{u}_n)$. Since $\{\bar{u}_n\}$ is bounded in $H^1(\mathbb{R}^N)$, upto subsequence $\{\left\| \nabla \bar{u}_n\right\|_2^2+s[\bar{u}_n]^2\}$ is convergent. Denoting the convergent subsequence as $\{\left\| \nabla \bar{u}_n\right\|_2^2+s[\bar{u}_n]^2\}$ itself, let $l\geq 0$, be such that
			\begin{equation}\label{3.6}
				l=\lim_{n\rightarrow \infty}\left(\left\| \nabla \bar{u}_n\right\|_2^2+s[\bar{u}_n]^2\right)=\lim_{n\rightarrow \infty}A(\bar{u}_n),
			\end{equation}
			%$$l=\lim_{n\rightarrow \infty}\left(\left\| \nabla \bar{u}_n\right\|_2^2+s[\bar{u}_n]^2\right)=\lim_{n\rightarrow \infty}A(u_n),$$
			then, by \eqref{S_alpha}, we have, either $l=0$ or $l\geq S_{\alpha}^{\frac{2^*_{\alpha}}{2^*_{\alpha}-1}}$.\\
			Subclaim: $l=0$.\\
			Let if possible, $l\geq S_{\alpha}^{\frac{2^*_{\alpha}}{2^*_{\alpha}-1}}$, then by \eqref{3.5}, Fatou's lemma and Gagliardo-Nirenberg inequality \eqref{G_N_inequality},
			\begin{eqnarray*}
				m_{\tau} & = & \lim_{n\rightarrow \infty}E(u_n)\\
				& = & \lim_{n\rightarrow \infty}\left(\frac{\left\| \nabla \bar{u}_n\right\|+\left\|\nabla u_0\right\|_2^2}{2}+\frac{[\bar{u}_n]^2+[u_0]^2}{2}-\mu \frac{\left\| u_n \right\|_p^p}{p}-\frac{A(\bar{u}_n)+A(u_0)}{22^*_{\alpha}}\right)\nonumber\\
				& \geq & \lim_{n\rightarrow \infty}\left(\frac{\left\| \nabla \bar{u}_n\right\|_2^2+s[\bar{u}_n]^2}{2}
				%-\mu\frac{\left\| \bar{u}_n\right\|_p^p}{p}
				-\frac{A(\bar{u}_n)}{22^*_{\alpha}}\right)+E(u_0)\nonumber\\
				& = & \left(\frac{2^*_{\alpha}-1}{22^*_{\alpha}}\right)l+E(u_0)\geq \left(\frac{2^*_{\alpha}-1}{22^*_{\alpha}}\right)S_{\alpha}^{\frac{2^*_{\alpha}}{2^*_{\alpha}-1}}+E(u_0)\nonumber\\
				& = & \left(\frac{2^*_{\alpha}-1}{22^*_{\alpha}}\right)S_{\alpha}^{\frac{2^*_{\alpha}}{2^*_{\alpha}-1}}+E(u_0) -\frac{M(u_0)}{22^*_{\alpha}}\nonumber\\
				%	& \geq & \left( \frac{2^*_{\alpha}-1}{22^*_{\alpha}}\right)\left\| \nabla  u_0 \right\|_2^2\\
				& \geq & \left(\frac{2^*_{\alpha}-1}{22^*_{\alpha}}\right)T(u_0)^2+\mu\left(\frac{p\gamma_p-22^*_{\alpha}}{22^*_{\alpha}p}\right)\left\| u_0 \right\|_p^p+\left(\frac{2^*_{\alpha}-1}{22^*_{\alpha}}\right)S_{\alpha}^{\frac{2^*_{\alpha}}{2^*_{\alpha}-1}}\nonumber\\
				& \geq &\left(\frac{2^*_{\alpha}-1}{22^*_{\alpha}}\right)T(u_0)^2+\mu\left(\frac{p\gamma_p-22^*_{\alpha}}{22^*_{\alpha}p}\right)C_{N,p}T(u_0)^{p\gamma_p}\tau^{p(1-\gamma_p)}+ \left(\frac{2^*_{\alpha}-1}{22^*_{\alpha}}\right)S_{\alpha}^{\frac{2^*_{\alpha}}{2^*_{\alpha}-1}}\nonumber\\
				& = & f(T(u_0))+\left(\frac{2^*_{\alpha}-1}{22^*_{\alpha}}\right)S_{\alpha}^{\frac{2^*_{\alpha}}{2^*_{\alpha}-1}},
			\end{eqnarray*}
			where $$f(t)=\left(\frac{2^*_{\alpha}-1}{22^*_{\alpha}}\right)t^2+\mu\left(\frac{p\gamma_p-22^*_{\alpha}}{22^*_{\alpha}p}\right)C_{N,p}t^{p\gamma_p}\tau^{p-p\gamma_p}.$$
			Now, since $t_0=\left(\frac{(22^*_{\alpha}-p\gamma_p)\gamma_p\mu C_{N,p}\tau^{p-p\gamma_p}}{2(2^*_{\alpha}-1)}\right)^{\frac{1}{2-p\gamma_p}}$ is the point of global minima of $f$. Thus,
			\begin{eqnarray*}
				m_{\tau} & \geq & f(t_0)+\left(\frac{2^*_{\alpha}-1}{22^*_{\alpha}}\right)S_{\alpha}^{\frac{2^*_{\alpha}}{2^*_{\alpha}-1}}\\
				& = & -\left(\frac{\gamma_p}{2^*_{\alpha}-1}\right)^{\frac{p\gamma_p}{2-p\gamma_p}}\left(\frac{2-p\gamma_p}{22^*_{\alpha}p}\right)\left(\frac{(22^*_{\alpha}-p\gamma_p)\mu C_{N,p}\tau^{p(1-\gamma_p)}}{2}\right)^{\frac{2}{2-p\gamma_p}}+\left(\frac{2^*_{\alpha}-1}{22^*_{\alpha}}\right)S_{\alpha}^{\frac{2^*_{\alpha}}{2^*_{\alpha}-1}}\\
				& > &0, \text{ for } \tau<\tau_1.
				%= \left(\frac{2(2^*_{\alpha}-1)}{\gamma_p^{\frac{p\gamma_p}{2}}\mu C_{N,p}(22^*_{\alpha}-p\gamma_p)}\right)^{\frac{1}{p\gamma_p}}\left(\frac{pS_{\alpha}^{\frac{2^*_{\alpha}}{2^*_{\alpha}-1}}}{2-p\gamma_p}\right)^{\frac{2-p\gamma_p}{2p(1-\gamma_p)}}.
			\end{eqnarray*}
			But this contradicts \autoref{Lemma 2.6}. Therefore $l=0$. Now, by \eqref{3.5} and \eqref{3.6},  $\displaystyle \lim_{n\rightarrow \infty}A(u_n)=A(u_0)$ and $\displaystyle \lim_{n\rightarrow \infty}T(u_n)\rightarrow T(u_0)$, then taking $u_0$ as test function in \eqref{3.4} and using \eqref{3.1} we get:
			\begin{eqnarray*}
				\lambda_0\left\| u_0 \right\|_2^2 & = & E'(u_0)(u_0)-\lim_{n\rightarrow \infty}\left(E'(u_n)(u_n)-\lambda_n\Phi'(u_n)(u_n)\right)= \lambda_0\lim_{n\rightarrow \infty}\left\| u_n \right\|_2^2=\lambda_0 \tau^2.
			\end{eqnarray*}
			Hence $u_0$ is a solution of \eqref{prob} and $u_n\rightarrow u_0$ strongly in $H^1(\mathbb{R}^N)$. Taking $u_{\tau}^+=u_0$ and $\lambda_{\tau}^+=\lambda_0$, we are done.
		\end{myproof}
		\section{Second Solution}
		Until now, we have seen that the infimum of $E$ on $\mathcal{M}_{\tau}^+$ is achieved and is a solution of \eqref{prob}. In this section, we will see that the infimum over $\mathcal{M}_{\tau}^-$, that is, $m_{\tau}^-$ is also achieved. Since the spaces $\mathcal{M}_{\tau}^+$ and $\mathcal{M}_{\tau}^-$ are disjoint, this corresponds to the second normalized solution. The following result will play a crucial role in proving the convergence of the Palaise Smale sequence, by providing us an upper bound for $m_{\tau}^-$. %will play a crucial role in proving the convergence of the Palaise Smale sequence.
		% Now, with the help of first solution, 
		% we can deduce the following relation  between $m_{\tau}^+$ and $m_{\tau}^-$, 
		%Let us start with discussing the following relation between $m_{\tau}^+$ and $m_{\tau}^-$.
		\begin{lemma}\label{Lemma 4.1}
			For all $\tau<\min\{\tau_0,\tau_1\}$,
			\begin{equation}\label{4.1}
				m_{\tau}^-=\inf_{u\in \mathcal{M}_{\tau}^-}E(u)<m_{\tau}+\left(\frac{2^*_{\alpha}-1}{22^*_{\alpha}}\right)S_{\alpha}^{\frac{2^*_{\alpha}}{2^*_{\alpha}-1}}.
			\end{equation}
		\end{lemma}
		\begin{proof}
			Let $\phi\in C_c^{\infty}(\mathbb{R}^N)$ be a cut-off function such that 
			\begin{equation}\label{4.2}
				\left\{
				\begin{array}{cl}
					0\leq \phi(x) \leq 1 & \text{ for all } x\in \mathbb{R}^N,\\
					\phi (x) = 1 & \text{ for } x\in B_1(0),\\
					\phi(x)=0  & \text{ for } x\in \mathbb{R}^N\setminus B_2(0),
				\end{array}
				\right.
			\end{equation}
			then, taking $u_{\epsilon}=\phi U_{\epsilon,0}$, where $U_{\epsilon,0}$ is as defined in \eqref{U_epsilon}, by \cite[lemma~1.46]{Willem2012Minimax}, \cite[lemma~3.3, eq~3.7]{Cassani2018Choquard} and \cite[lemma~5.3]{da2024MIxed} we have:
			\begin{equation}\label{grad_u_epsilon}
				\left\| \nabla u_{\epsilon}\right\|_2^2 = S^{\frac{N}{2}}+O(\epsilon^{N-2}),
			\end{equation}
			\begin{equation}\label{u_epsilon}
				\left\| u_{\epsilon}\right\|_2^2 = \left\{
				\begin{array}{ll}
					K_1\epsilon^2+O(\epsilon^{N-2}) & \text{ for } N\geq 5,\\
					K_1\epsilon^2|\ln(\epsilon)|+O(\epsilon^2) & \text{ for } N=4,\\
					K_1\epsilon +O(\epsilon^2) & \text{ for } N=3,	
				\end{array}
				\right.
			\end{equation}
			\begin{equation}\label{A(u_epsilon)}
				A(u_{\epsilon}) \geq (A_{\alpha}C_{\alpha})^{\frac{N}{2}}S_{\alpha}^{\frac{N+\alpha}{2}}-O(\epsilon^{\frac{N+\alpha}{2}}),
			\end{equation}
			\begin{equation}\label{[U_epsilon]}
				[u_{\epsilon}]^2 = O(\epsilon^{m_{N,s}}) \text{ where }
				m_{N,s}=\left\{\begin{array}{cl}
					2(1-s) & \text{ for } N\geq 4 \text{ and } N=3 \text{ with } s>\frac{1}{2},\\
					1 & \text{ for } N=3 \text{ with } s\leq\frac{1}{2}.
				\end{array}\right.
			\end{equation}
			and 
			\begin{comment}
				Now, for $N>\frac{2p}{p-1}$
				\begin{eqnarray*}
					\left\| u_{\epsilon} \right\|_p^p & = & \int_{B_1(0)}|U_{\epsilon,0}|^p+\int_{\mathbb{R}^N\setminus B_1(0)}|\phi|^p|U_{\epsilon,0}|^p= \int_{B_1(0)}|U_{\epsilon,0}|^p+O(\epsilon^{\frac{p(N-2)}{2}})\\
					& = & \int_{\mathbb{R}^N}|U_{\epsilon,0}|^p+O(\epsilon^{\frac{p(N-2)}{2}})= K_2'\epsilon^{N-\frac{p(N-2)}{2}}+O(\epsilon^{\frac{p(N-2)}{2}}),
				\end{eqnarray*}
				for $N=\frac{2p}{p-1}$,
				\begin{eqnarray*}
					\left\| u_{\epsilon} \right\|_p^p & = & \int_{B_1(0)}|U_{\epsilon,0}|^p+\int_{\mathbb{R}^N\setminus B_1(0)}|\phi|^p|U_{\epsilon,0}|^p =\int_{B_1(0)}|U_{\epsilon,0}|^p+O(\epsilon^{\frac{N}{2}})\\
					& = & K_2''\epsilon^{\frac{N}{2}}\ln(1/\epsilon)+O(\epsilon^{\frac{N}{2}}),
				\end{eqnarray*}
				and, for $N<\frac{2p}{p-1}$
				\begin{equation}
					\left\| u_{\epsilon} \right\|_p^p  =  \int_{B_1(0)}|U_{\epsilon,0}|^p+\int_{\mathbb{R}^N\setminus B_1(0)}|\phi|^p|U_{\epsilon,0}|^p\leq K_2'''\epsilon^{\frac{p(N-2)}{2}},
				\end{equation}
				taking $K_2=\max\{K_2',K_2''\}$, we get:
			\end{comment}
			\begin{equation}\label{U_epsilon_p}
				\left\| u_{\epsilon}\right\|_p^p=\left\{
				\begin{array}{ll}
					K_2 \epsilon^{N-\frac{p(N-2)}{2}}+O(\epsilon^{\frac{p(N-2)}{2}}) & \text{ for } N>\frac{2p}{p-1},\\
					K_2\epsilon^{\frac{N}{2}}\ln(1/\epsilon)+O(\epsilon^{\frac{N}{2}}) & \text{ for } N=\frac{2p}{p-1},\\
					O(\epsilon^{\frac{p(N-2)}{2}}) & \text{ for } N<\frac{2p}{p-1}.
				\end{array}
				\right.
			\end{equation}
			For $\zeta, t\geq 0$, define 
			$$\hat{u}_{\epsilon,t}(x):=u_{\tau}^+(x)+tu_{\epsilon}(x);\text{ and }\bar{u}_{\epsilon,t}(x):=\zeta^{\frac{N-2}{2}}\hat{u}(\zeta x),$$ 
			with $u_{\tau}^+$ being the radial solution deduced in \autoref{Theorem 1}. We will see that $\displaystyle m_{\tau}^-\leq \sup_{t\geq 0}E(\bar{u}_{\epsilon,t})$ and $E(\bar{u}_{\epsilon,t})<m_{\tau}+\left(\frac{2^*_{\alpha}-1}{22^*_{\alpha}}\right)S_{\alpha}^{\frac{2^*_{\alpha}}{2^*_{\alpha}-1}}$ for all $t>0$ and small enough $\epsilon>0$.
			Clearly,
			\begin{equation}\label{4.9}
				\left\{
				\begin{array}{c}
					\left\| \nabla \bar{u}_{\epsilon,t} \right\|_2^2=\left\| \nabla\hat{u}_{\epsilon,t}\right\| _2^2; \;\;
					[\bar{u}_{\epsilon,t}]^2 = \zeta^{2(s-1)}[\hat{u}_{\epsilon,t}]^2;\;\;\left\| \bar{u}_{\epsilon,t}\right\|_2^2=\zeta^{-2}\left\| \hat{u}_{\epsilon,t} \right\|_2^2\\
					\left\| \bar{u}_{\epsilon,t} \right\|_p^p=\zeta^{p\gamma_p-p}\left\| \hat{u}_{\epsilon,t}\right\|_p^p;\;\; A(\bar{u}_{\epsilon,t})= A(\hat{u}_{\epsilon,t}),
				\end{array}
				\right.
			\end{equation}
			then, taking $\zeta=\zeta_{\epsilon,t}=\frac{\left\| \hat{u}_{\epsilon,t}\right\|_2}{\tau}$, we get $\bar{u}_{\epsilon,t}\in S(\tau)$. %Now, since $\bar{u}_{\epsilon,t}\in S(\tau)$,
			Thus by \autoref{Lemma 2.3}, we can find $\bar{q}_{\epsilon,t}\in \mathbb{R}$ such that $\bar{q}_{\epsilon,t}\star \bar{u}_{\epsilon,t}\in \mathcal{M}_{\tau}^-$, or, $q_{\epsilon,t}\circledast \bar{u}_{\epsilon,t}\in \mathcal{M}_{\tau}^-$ where $q_{\epsilon,t}=e^{\bar{q}_{\epsilon,t}}>0$. Then,
			\begin{equation*}
				0  =  M(q_{\epsilon,t}\circledast \bar{u}_{\epsilon,t}) = q_{\epsilon,t}^2 \left\| \nabla \bar{u}_{\epsilon,t} \right\|_2^2+s q_{\epsilon,t}^{2s}[\bar{u}_{\epsilon,t}]^2-\mu\gamma_p q_{\epsilon,t}^{p\gamma_p}\left\| \bar{u}_{\epsilon,t}\right\|_p^p-q_{\epsilon,t}^{22^*_{\alpha}}A(\bar{u}_{\epsilon,t}),
			\end{equation*}
			and hence,
			\begin{equation}\label{4.10}
				q_{\epsilon,t}^{2-p\gamma_p} \left\| \nabla \bar{u}_{\epsilon,t} \right\|_2^2+s q_{\epsilon,t}^{2s-p\gamma_p}[\bar{u}_{\epsilon,t}]^2=\mu\gamma_p \left\| \bar{u}_{\epsilon,t}\right\|_p^p+q_{\epsilon,t}^{22^*_{\alpha}-p\gamma_p}A(\bar{u}_{\epsilon,t}).
			\end{equation}
			Now, since $0\star \hat{u}_{\epsilon,0}=u_{\tau}^+\in \mathcal{M}_{\tau}^+$, by \autoref{Lemma 2.3}, $\bar{q}_{\epsilon,0}>0$, that is, $q_{\epsilon,0}>1$. Also, by \eqref{4.10}
			\begin{equation*}
				q_{\epsilon,t}^{22^*_{\alpha}} \leq \frac{q_{\epsilon,t}^2 \left\| \nabla\bar{u}_{\epsilon,t}\right\|_2^2+sq_{\epsilon,t}^{2s}[\bar{u}_{\epsilon,t}]^2}{A(\bar{u}_{\epsilon,t})},
			\end{equation*}
			defining $B_{\epsilon,t}:=\frac{\left\| \nabla \bar{u}_{\epsilon,t}\right\|_2^2+s[\bar{u}_{\epsilon,t}]^2}{A(\bar{u}_{\epsilon,t})}$, we get $0<q_{\epsilon,t}\leq \max\{B_{\epsilon,t}^{\frac{1}{2(2^*_{\alpha}-1)}}, B_{\epsilon,t}^{\frac{1}{2(2^*_{\alpha}-s)}}\}$. By \eqref{4.9} we have:
			\begin{eqnarray*}
				B_{\epsilon,t} & = & \frac{\left\| \nabla \bar{u}_{\epsilon,t}\right\|_2^2+s[\bar{u}_{\epsilon,t}]^2}{A(\bar{u}_{\epsilon,t})}=\frac{\left\|\nabla \hat{u}_{\epsilon,t}\right\|_2^2+s\zeta_{\epsilon,t}^{2(s-1)}[\hat{u}_{\epsilon,t}]^2}{A(\hat{u}_{\epsilon,t})}\\
				& = &\frac{1}{A(\hat{u}_{\epsilon,t})}\left(\left\|\nabla \hat{u}_{\epsilon,t}\right\|_2^2+s\left(\frac{\tau}{\left\| \hat{u}_{\epsilon,t}\right\|_2^2}\right)^{2(1-s)}[\hat{u}_{\epsilon,t}]^2\right)\leq \frac{\left\| \nabla \hat{u}_{\epsilon,t}\right\|_2^2+s[\hat{u}_{\epsilon,t}]^2}{A(\hat{u}_{\epsilon,t})}\\
				& \leq & C\left(\frac{\left\| \nabla u_{\tau}^+\right\|_2^2+t^2\left\| \nabla u_{\epsilon}\right\|_2^2+s[u_{\tau}^+]^2+st^2[u_{\epsilon}]^2}{t^{22^*_{\alpha}}A(u_{\epsilon})}\right)\rightarrow 0 \text{ as } t\rightarrow\infty,
			\end{eqnarray*}
			and hence $q_{\epsilon,t}\rightarrow 0$ as $t\rightarrow\infty$. 
			Since $q_{\epsilon,0}>1$, there exists some $t_{\epsilon}>0$ such that $q_{\epsilon,t_{\epsilon}}=1$, which implies that
			\begin{equation}\label{4.11}
				m_{\tau}^-=\inf_{u\in \mathcal{M}_{\tau}^-}E(u)\leq E(q_{\epsilon,t_{\epsilon}}\circledast \bar{u}_{\epsilon,t_{\epsilon}})= E(\bar{u}_{\epsilon,t_{\epsilon}})\leq \sup_{t\geq 0}E(\bar{u}_{\epsilon,t}). 
			\end{equation}
			Now, since $\hat{u}_{\epsilon,t}\geq u_{\tau}^+$ by \eqref{4.9} and definition of $\bar{u}_{\epsilon,t}$, we have:
			\begin{eqnarray}\label{4.12}
				E(\bar{u}_{\epsilon,t}) & = & \frac{\left\| \nabla u_{\tau}^++t\nabla u_{\epsilon}\right\|_2^2}{2}+\frac{\zeta_{\epsilon,t}^{2(s-1)}}{2}[u_{\tau}^++tu_{\epsilon}]^2-\frac{\zeta_{\epsilon,t}^{p(\gamma_p-1)}\mu}{p}\left\| u_{\tau}^++tu_{\epsilon}\right\|_p^p-\frac{A(u_{\tau}^++tu_{\epsilon})}{22^*_{\alpha}}\nonumber\\		
				& \leq & \frac{\left\| \nabla u_{\tau}^+\right\|_2^2}{2}+\frac{t^2\left\| \nabla u_{\epsilon}\right\|_2^2}{2}+t\int_{\mathbb{R}^N}\nabla u_{\tau}^+\nabla u_{\epsilon} +\frac{[u_{\tau}^+]^2}{2}+\frac{t^2[u_{\epsilon}]^2}{2}+t\ll u_{\tau}^+, u_{\epsilon}\gg\nonumber\\
				&& -\mu\frac{\left\| u_{\tau}^+\right\|_p^p}{p}-\frac{A(u_{\tau}^+)}{22^*_{\alpha}} \nonumber\\
				& = & E(u_{\tau}^+)+\frac{t^2\left\| \nabla u_{\epsilon}\right\|_2^2}{2}+t\int_{\mathbb{R}^N}\nabla u_{\tau}^+\nabla u_{\epsilon}+\frac{t^2[u_{\epsilon}]^2}{2}+t\ll u_{\tau}^+, u_{\epsilon}\gg\nonumber\\
				&& \rightarrow E(u_{\tau}^+) =m_{\tau} <0 \text{ as } t\rightarrow0^+.
			\end{eqnarray}
			Also, 
			%since $u_{\tau}^+\in \mathcal{M}_{\tau}$ solves \eqref{prob}, we have:
			%$$\lambda_{\tau}^+\tau^2=(1-s)[u_{\tau}^+]^2+(\gamma_p-1)\left\| u_{\tau}^+\right\|_p^p,$$
			%and hence
			\begin{eqnarray}\label{4.13}
				E(\bar{u}_{\epsilon,t}) & \leq & \frac{\left\| \nabla u_{\tau}^+\right\|_2^2}{2}+\frac{[u_{\tau}^+]^2}{2}-\mu\frac{\left\| u_{\tau}^+\right\|_p^p}{p}+t^2\frac{\left\| \nabla u_{\epsilon} \right\|_2^2}{2}+t^2\frac{[u_{\epsilon}]^2}{2}+t\int_{\mathbb{R}^N}\nabla u_{\tau}^+\nabla u_{\epsilon}\nonumber\\
				&& + t\ll u_{\tau}^+,u_{\epsilon}\gg-\frac{t^{22^*_{\alpha}}}{22^*_{\alpha}}A(u_{\epsilon})\nonumber\\
				&&\rightarrow -\infty \text{ as } t\rightarrow+\infty,
			\end{eqnarray}
			and by \autoref{Lemma 2.3}, $E(\bar{u}_{\epsilon,t_{\epsilon}})=E(0\star \bar{u}_{\epsilon,t_{\epsilon}})=E(\bar{q}_{\epsilon,t_{\epsilon}}\star \bar{u}_{\epsilon,t_{\epsilon}})>0$, thus there exists some $t_0>0$ large enough such that $E(\bar{u}_{\epsilon,t})<0$ for $t\in (0,\frac{1}{t_0})\cup (t_0,\infty)$. Therefore, we need to estimate $E(\bar{u}_{\epsilon,t})$ in $[\frac{1}{t_0}, t_0]$. Above analysis can be summerized by the following plot:
			
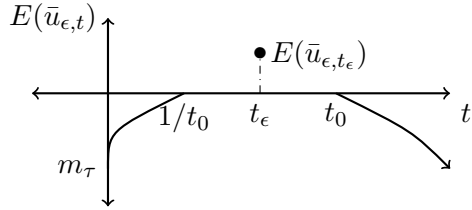
\begin{center}
				\begin{tikzpicture}
					\draw[thick,<->] (-1,0)--(4.5,0) node[anchor=north west]{$t$};
					\draw[thick,<->] (0,-1.5)--(0,1) node[anchor=east]{$E(\bar{u}_{\epsilon,t})$};
					\draw[thick] (1,0).. controls (0,-0.5)..(0,-1) node[anchor=east]{$m_{\tau}$};
					\draw[thick,->] (3,0)..controls(4,-0.5)..(4.5,-1);
					\draw [thick](2,0) node[anchor= north]{$t_{\epsilon}$};
					\draw [dash dot](2,0)--(2,0.5) node[anchor=west]{$E(\bar{u}_{\epsilon,t_{\epsilon}})$};
					\draw (2,0.75) node[anchor=north]{$\bullet$};
					\draw (1,0) node[anchor=north]{$1/t_0$};
					\draw (3,0) node[anchor=north]{$t_0$};
				\end{tikzpicture}
			\end{center}
			Now, let us study $E(\bar{u}_{\epsilon,t})$ for $t\in [1/t_0,t_0]$. Since,
			%Clearly, for $\frac{1}{t_0}\leq t \leq t_0$,
			\begin{equation*}
				\zeta_{\epsilon,t}^2 =  \frac{\left\| \hat{u}_{\epsilon,t}\right\|_2^2}{\tau^2}=1+\frac{t^2}{\tau^2}\int_{\mathbb{R}^N}|u_{\epsilon}|^2+\frac{2t}{\tau^2}\int_{\mathbb{R}^N}u_{\tau}^+u_{\epsilon},
			\end{equation*}
			and hence, 
			\begin{eqnarray*}
				\zeta_{\epsilon,t}^{p\gamma_p-p} & = &  \left(1+\left(\frac{t^2}{\tau^2}\int_{\mathbb{R}^N}|u_{\epsilon}|^2+\frac{2t}{\tau^2}\int_{\mathbb{R}^N}u_{\tau}^+u_{\epsilon}\right)\right)^{\frac{p(\gamma_p-1)}{2}}\\
				& \geq &1+\frac{p(\gamma_p-1)}{2}\left(\frac{t^2}{\tau^2}\int_{\mathbb{R}^N}|u_{\epsilon}|^2+\frac{2t}{\tau^2}\int_{\mathbb{R}^N}u_{\tau}^+u_{\epsilon}\right),
			\end{eqnarray*}
			by \eqref{4.9} and the fact that $\hat{u}_{\epsilon,t}\geq u_{\tau}^+$, we get
			\begin{eqnarray}\label{4.14}
				E(\bar{u}_{\epsilon,t}) & \leq & \frac{\left\| \nabla \hat{u}_{\epsilon,t}\right\|_2^2}{2}+\frac{[\hat{u}_{\epsilon,t}]^2}{2}-\frac{A(\hat{u}_{\epsilon,t})}{22^*_{\alpha}}\nonumber\\
				& & -\left(1+\frac{p(\gamma_p-1)}{2}\left(\frac{t^2}{\tau^2}\int_{\mathbb{R}^N}|u_{\epsilon}|^2+\frac{2t}{\tau^2}\int_{\mathbb{R}^N}u_{\tau}^+u_{\epsilon}\right)\right)\frac{\mu\left\| \hat{u}_{\epsilon,t}\right\|_p^p}{p}.
			\end{eqnarray}
			Further, we have:
			\begin{equation}\label{4.15}
				A(\hat{u}_{\epsilon,t}) = A(u_{\tau}^++tu_{\epsilon}) \geq A(u_{\tau}^+)+A(tu_{\epsilon}) +22^*_{\alpha}\int_{\mathbb{R}^N}(I_{\alpha}*|u_{\tau}^+|^{2^*_{\alpha}})|u_{\tau}^+|^{2^*_{\alpha}-2}u_{\tau}^+(tu_{\epsilon}),
			\end{equation}
			%	because $(a+b)^\beta(c+d)^{\beta}\geq a^{\beta}c^{\beta}+b^{\beta}d^{\beta}+\beta c^{\beta-1}da^{\beta}+\beta a^{\beta-1}bc^{\beta}$ for all $a,b,c,d \geq0$, $\beta>0$, also
			and
			\begin{equation}\label{4.16}
				\left\| \hat{u}_{\epsilon,t} \right\|_p^p \geq \left\| u_{\tau}^+\right\|_p^p+\left\| tu_{\epsilon}\right\|_p^p= \left\| u_{\tau}^++tu_{\epsilon}\right\|_p^p\geq \left\| u_{\tau}^+\right\|_p^p+\left\| tu_{\epsilon}\right\|_p^p+pt\int_{\mathbb{R}^N}|u_{\tau}^+|^{p-2}u_{\tau}^+u_{\epsilon},
			\end{equation}
			thus, using \eqref{4.15} and \eqref{4.16} in \eqref{4.14} 
			\begin{eqnarray*}
				E(\bar{u}_{\epsilon,t}) & \leq & E(u_{\tau}^+)+E(tu_{\epsilon})+\left(t\int_{\mathbb{R}^N}\nabla u_{\tau}^+\nabla u_{\epsilon}+t\ll u_{\tau}^+,u_{\epsilon}\gg-t\mu\int_{\mathbb{R}^N}|u_{\tau}^+|^{p-2}u_{\tau}^+u_{\epsilon}\right.\\
				&& \left. -\int_{\mathbb{R}^N}(I_{\alpha}*|u_{\tau}^+|^{2^*_{\alpha}})|u_{\tau}^+|^{2^*_{\alpha}-2}u_{\tau}^+(tu_{\epsilon})\right)+\frac{(1-\gamma_p)t^2}{2\tau^2}\mu\left\| u_{\epsilon}\right\|_2^2\left\| \hat{u}_{\epsilon,t}\right\|_p^p\\
				&& +\mu\frac{t(1-\gamma_p)}{\tau^2}\left\| \hat{u}_{\epsilon,t}\right\|_p^p\int_{\mathbb{R}^N}u_{\tau}^+u_{\epsilon},
			\end{eqnarray*}
			moreover, since $u_{\tau}^+$ solves \eqref{prob}, we get:
			\begin{eqnarray}\label{4.17}
				E(\bar{u}_{\epsilon,t}) & \leq & E(u_{\tau}^+)+E(tu_{\epsilon})+\lambda_{\tau}^+\int_{\mathbb{R}^N}u_{\tau}^+(tu_{\epsilon})+\frac{\mu(1-\gamma_p)t^2}{2\tau^2}\left\| u_{\epsilon}\right\|_2^2\left\| \hat{u}_{\epsilon,t}\right\|_p^p\nonumber\\
				&& +\frac{\mu t(1-\gamma_p)}{\tau^2}\left\| \hat{u}_{\epsilon,t}\right\|_p^p\int_{\mathbb{R}^N}u_{\tau}^+u_{\epsilon}\nonumber\\
				& = & m_{\tau}+E(tu_{\epsilon})+\frac{\mu t(1-\gamma_p)}{\tau^2}\left(\left\| \hat{u}_{\epsilon,t}\right\|_p^p-\left\| u_{\tau}^+\right\|_p^p\right)\int_{\mathbb{R}^N}u_{\tau}^+u_{\epsilon}\nonumber\\
				&& +\frac{\mu t^2(1-\gamma_p)}{2\tau^2}\left\| u_{\epsilon}\right\|_2^2\left\| \hat{u}_{\epsilon,t}\right\|_p^p+\frac{t(1-s)}{\tau^2}[u_{\tau}^+]^2\int_{\mathbb{R}^N}u_{\tau}^+u_{\epsilon}.
			\end{eqnarray}
			Since $u_{\tau}^+$ is a radially symmetric solution of \eqref{prob}, as done in \cite[lemma~5.5]{Jeanjean2022multiple} one can deduce that :
			$$\int_{\mathbb{R}^N}u_{\tau}^+u_{\epsilon}=O(\epsilon^{\frac{N-2}{2}}); \text{ and } \int_{\mathbb{R}^N}|u_{\tau}^+|^{p-1}u_{\epsilon}=O(\epsilon^{\frac{N-2}{2}}),$$
			then \eqref{4.17} becomes:
			\begin{eqnarray*}
				E(\bar{u}_{\epsilon,t}) & \leq & m_{\tau}+E(tu_{\epsilon})+\frac{\mu t(1-\gamma_p)}{\tau^2}\left(O(\epsilon^{\frac{N-2}{2}})+\left\| tu_{\epsilon}\right\|_p^p\right)O(\epsilon^{\frac{N-2}{2}})\\
				&& +\frac{\mu t^2(1-\gamma_p)}{2\tau^2}\left\| u_{\epsilon}\right\|_2^2\left\| u_{\tau}^++tu_{\epsilon}\right\|_p^p+\frac{t(1-s)}{\tau^2}[u_{\tau}^+]O(\epsilon^{\frac{N-2}{2}})\\
				& = & m_{\tau} +E(tu_{\epsilon}) +O(\epsilon^{N-2})+O(\left\| u_{\epsilon}\right\|_p^p)O(\epsilon^{\frac{N-2}{2}}) +O(\left\| u_{\epsilon}\right\|_2^2)\\
				&& +O(\left\| u_{\epsilon}\right\|_2^2)O(\left\| u_{\epsilon}\right\|_p^p)+O(\epsilon^{\frac{N-2}{2}})\\
				& \leq & m_{\tau} +f_{u_{\epsilon}}(t) \text{ for small } \epsilon>0,\text{ where }f_{u}(t):=\frac{t^2T(u)^2}{2}-\frac{t^{22^*_{\alpha}}A(u)}{22^*_{\alpha}},
			\end{eqnarray*}
			also, since $f_u$ has global maxima at $t_u=\left(\frac{T(u)^2}{A(u)}\right)^{\frac{1}{2(2^*_{\alpha}-1)}}$, by \eqref{grad_u_epsilon}, \eqref{[U_epsilon]} and \eqref{A(u_epsilon)} we get:
			\begin{eqnarray*}
				E(\bar{u}_{\epsilon,t}) & \leq & m_{\tau} +f_{u_{\epsilon}}(t_{u_{\epsilon}}) =m_{\tau}+\left(\frac{2^*_{\alpha}-1}{22^*_{\alpha}}\right)\left(\frac{T(u_{\epsilon})^2}{A(u_{\epsilon})^{\frac{1}{2^*_{\alpha}}}}\right)^{\frac{2^*_{\alpha}}{2^*_{\alpha}-1}}\\
				& \leq & m_{\tau} +  \left(\frac{2^*_{\alpha}-1}{22^*_{\alpha}}\right)\left(\frac{S^{\frac{N}{2}}+O(\epsilon^{N-2})+O(\epsilon^{m_{N,s}})}{\left((A_{\alpha}C_{\alpha})^{\frac{N}{2}}S_{\alpha}^{\frac{N+\alpha}{2}}-O(\epsilon^{\frac{N+\alpha}{2}})\right)^{\frac{1}{2^*_{\alpha}}}}\right)^{\frac{2^*_{\alpha}}{2^*_{\alpha}-1}}\\
				& < & m_{\tau}+\left(\frac{2^*_{\alpha}-1}{22^*_{\alpha}}\right)S_{\alpha}^{\frac{2^*_{\alpha}}{2^*_{\alpha}-1}} \text{ as } \epsilon \text{ goes to zero, for all } t\in [1/t_0,t_0],
			\end{eqnarray*}
			therefore, by \eqref{4.11} we are done. 
			%				Now, for $n\geq 5$, by \eqref{grad_u_epsilon}, \eqref{[U_epsilon]}, \eqref{A(u_epsilon)} and \eqref{U_epsilon_p} we have:
			%			\begin{equation}
				%		E(\bar{u}_{\epsilon,t}) \leq m_{\tau}+\frac{t^2}{2}S^{\frac{N}{2}}-\frac{t^{22^*_{\alpha}}}{22^*_{\alpha}}(A_{\alpha}C_{\alpha})S^{\frac{N+\alpha}{2}} \text{ as }\epsilon\rightarrow 0.
				%	\end{equation}
		\end{proof}
		\noindent For $0<\tau<\min\{\tau_0,\tau_1\}$, let $u\in \mathcal{M}_{\tau}^{\pm}$, then $v_{\beta}:=\frac{\beta}{\tau}u\in S(\beta)$, for all $\beta>0$. Now, for $0<\beta<\min\{\tau_0,\tau_1\}$  by \autoref{Lemma 2.3}, there exists $t_{\pm}(\beta)>0$ such that $t_{\pm}(\beta)\circledast v_{\beta}\in \mathcal{M}_{\beta}^{\pm}$.  Clearly, since $v_{\tau}=u\in \mathcal{M}_{\tau}^{\pm}$, $t_{\pm}(\tau)=1$. Further, we have following results for $t_{\pm}(\beta)$.
		\begin{lemma}\label{Lemma 4.2}
			For $N\geq 3$, $2<p<2+\frac{4s}{N}$ and $0<\tau<\min\{\tau_0,\tau_1\}$, $t_{\pm}$ is differentiable at $\tau$, with
			\begin{equation*}
				t_{\pm}'(\tau) =\frac{p\gamma_p\mu\left\| u \right\|_p^p+22^*_{\alpha}A(u)-2s[u]^2-2\left\| \nabla u \right\|_2^2}{\tau\left(2s^2[u]^2+2\left\| \nabla u \right\|_2^2-\mu p\gamma_p^2\left\| u \right\|_p^p-22^*_{\alpha}A(u)\right)},
			\end{equation*}
			Moreover, for sufficiently large $\mu>0$, $E(t_{\pm}(\beta)\circledast v_{\beta})<E(u)$ whenever $\tau<\beta<\min\{\tau_0,\tau_1\}$.
		\end{lemma}
		\begin{proof}
			Since $M(t_{\pm}(\beta)\circledast v_{\beta})=0$ and $v_{\beta}=\frac{\beta}{\tau}u$, for all $0<\beta<\min\{\tau_0,\tau_1\}$,
			\begin{equation*}
				0 = 
				% M((t_{\pm}(\beta)\circledast v_{\beta}) = 
				\left(\frac{\beta t_{\pm}(\beta)}{\tau}\right)^2\left\| \nabla u \right\|_2^2+s\left(\frac{\beta t_{\pm}^s(\beta)}{\tau}\right)^2[u]^2-\mu\gamma_p\left(\frac{\beta t_{\pm}^{\gamma_p}(\beta)}{\tau}\right)^p\left\| u \right\|_p^p-\left(\frac{\beta t_{\pm}(\beta)}{\tau}\right)^{22^*_{\alpha}}A(u).
			\end{equation*}
			Defining $\Phi: (0,\min\{\tau_0,\tau_1\})\times (0,\infty)\rightarrow \mathbb{R}$ as follows:
			$$\Phi(\beta, t):=				\left(\frac{\beta t}{\tau}\right)^2\left\| \nabla u \right\|_2^2+s\left(\frac{\beta t^s}{\tau}\right)^2[u]^2-\mu\gamma_p\left(\frac{\beta t^{\gamma_p}}{\tau}\right)^p\left\| u \right\|_p^p-\left(\frac{\beta t}{\tau}\right)^{22^*_{\alpha}}A(u),$$
			we get, $\Phi(\beta,t_{\pm}(\beta))=0$, for all $0<\beta<\min\{\tau_0,\tau_1\}$ and since $u\in \mathcal{M}_{\tau}^{\pm}$, we have:
			$$\frac{\partial}{\partial t}\Phi(\tau,1)= 2s^2[u]^2+2\left\| \nabla u \right\|_2^2-\mu p\gamma_p^2\left\| u \right\|_p^p-22^*_{\alpha}A(u)\neq 0,$$
			thus, by implicit function theorem $\beta\mapsto t_{\pm}(\beta)$ is differentiable at $\tau$ and 
			\begin{equation*}
				t_{\pm}'(\tau)=-\frac{\frac{\partial }{\partial \beta}\Phi(\tau,1)}{\frac{\partial}{\partial t}\Phi(\tau,1)}=\frac{\mu p\gamma_p\left\| u \right\|_p^p+22^*_{\alpha}A(u)-2s[u]^2-2\left\| \nabla u \right\|_2^2}{\tau\left(2s^2[u]^2+2\left\| \nabla u \right\|_2^2-\mu p\gamma_p^2\left\| u \right\|_p^p-22^*_{\alpha}A(u)\right)},
			\end{equation*}
			hence
			\begin{equation}\label{4.18}
				1+\tau t_{\pm}'(\tau)=\frac{2s(s-1)[u]^2+\mu p\gamma_p(1-\gamma_p)\left\| u \right\|_p^p}{2s^2[u]^2+2\left\| \nabla u \right\|_2^2-\mu p\gamma_p^2\left\| u \right\|_p^p-22^*_{\alpha}A(u)}.
			\end{equation}
			%			 $$1+\tau t_{\pm}'(\tau)=\frac{2s(s-2)[u]^2+p\gamma_p(1-\gamma_p)\left\| u \right\|_p^p}{2s^2[u]^2+2\left\| \nabla u \right\|_2^2-p\gamma_p^2\left\| u \right\|_p^p-22^*_{\alpha}A(u)}$$
			Now, 
			\begin{eqnarray*}
				E(t_{\pm}(\beta)\circledast v_{\beta}) & = & \left(\frac{1}{2}-\frac{1}{p\gamma_p}\right)\left\| \nabla t_{\pm}(\beta)\circledast v_{\beta} \right\|_2^2+\left(\frac{1}{2}-\frac{s}{p\gamma_p}\right)[t_{\pm}(\beta)\circledast v_{\beta}]^2\\
				&& +\left(\frac{1}{p\gamma_p}-\frac{1}{22^*_{\alpha}}\right)A(t_{\pm}(\beta)\circledast v_{\beta})\\
				& = & \left(\frac{1}{2}-\frac{1}{p\gamma_p}\right)\left(\frac{t_{\pm}(\beta)\beta}{\tau}\right)^2\left\| \nabla u \right\|_2^2+\left(\frac{1}{2}-\frac{s}{p\gamma_p}\right)\left(\frac{t_{\pm}^{s}(\beta)\beta}{\tau}\right)^2[u]^2\\
				&& +\left(\frac{1}{p\gamma_p}-\frac{1}{22^*_{\alpha}}\right)\left(\frac{t_{\pm}(\beta)\beta}{\tau}\right)^{22^*_{\alpha}}A(u)\\
				&= &  \left(\frac{1}{2}-\frac{1}{p\gamma_p}\right)\left(1+(\beta-\tau)\left(\frac{t_{\pm}(\beta)\beta-\tau t_{\pm}(\tau)}{\tau(\beta-\tau)}\right)\right)^2\left\| \nabla u \right\|_2^2\\
				&& +\left(\frac{1}{2}-\frac{s}{p\gamma_p}\right)\left(1+(\beta-\tau)\left(\frac{t_{\pm}^s(\beta)\beta-\tau t_{\pm}^s(\tau)}{\tau(\beta-\tau)}\right)\right)^2[u]^2\\
				&& +\left(\frac{1}{p\gamma_p}-\frac{1}{22^*_{\alpha}}\right)\left(1+(\beta-\tau)\left(\frac{t_{\pm}(\beta)\beta-\tau t_{\pm}(\tau)}{\tau(\beta-\tau)}\right)\right)^{22^*_{\alpha}}A(u)\\
				& = & \left(\frac{1}{2}-\frac{1}{p\gamma_p}\right)\left\| \nabla u \right\|_2^2+\left(\frac{1}{2}-\frac{s}{p\gamma_p}\right)[u]^2+\left(\frac{1}{p\gamma_p}-\frac{1}{22^*_{\alpha}}\right)A(u) +o(\beta-\tau)^2\\
				&& +\left(2\frac{(1+\tau t_{\pm}'(\tau))}{\tau}\left(\frac{1}{2}-\frac{1}{p\gamma_p}\right)\left\| \nabla u \right\|_2^2+2\frac{(1+s\tau t_{\pm}'(\tau))}{\tau}\left(\frac{1}{2}-\frac{s}{p\gamma_p}\right) [u]^2\right.\\
				&& \left. 22^*_{\alpha}\frac{(1+\tau t_{\pm}'(\tau))}{\tau}\left(\frac{1}{p\gamma_p}-\frac{1}{22^*_{\alpha}}\right)A(u)\right)(\beta-\tau),
			\end{eqnarray*}
			further, since $M(u)=0$, one can deduce that
			\begin{eqnarray*}
				E(t_{\pm}(\beta)\circledast v_{\beta}) & = & \frac{2(\beta-\tau)}{\tau}\left(\gamma_p(1+\tau t_{\pm}'(\tau))\left(\frac{1}{2}-\frac{1}{p\gamma_p}\right)\left\| u \right\|_p^p+ \frac{(2^*_{\alpha}-1)(1+\tau t_{\pm}'(\tau))}{p\gamma_p}A(u)\right.\\
				&& \left. \frac{(1-s)s\tau t_{\pm}'(\tau)}{p\gamma_p}[u]^2\right)+ E(u)  +o(\beta-\tau)^2,
			\end{eqnarray*}
			and hence, by \eqref{4.18}
			$$E(t_{\pm}(\beta)\circledast v_{\beta}) = E(u)-\mu \frac{(1-\gamma_p)(\beta-\tau)}{\tau}\left\| u \right\|_p^p+\frac{(\beta-\tau)(1-s)}{\tau}[u]^2+o(\beta-\tau)^2.$$
			For sufficiently large $\mu>0$, we have:
			$$\frac{\partial}{\partial\beta}E(t_{\pm}(\beta)\circledast v_{\beta})_{|\beta=\tau}=-\frac{\mu (1-\gamma_p)\left\| u \right\|_p^p}{\tau}+\frac{(1-s)}{\tau}[u]^2<0,$$
			thus for $\tau<\beta<\min\{\tau_0,\tau_1\}$, $E(t_{\pm}(\beta)\circledast v_{\beta})<E(u)$.
		\end{proof}
			\noindent Denoting $\mathcal{M}_{r,\tau}^-:=\mathcal{M}_{\tau}^-\cap H_r(\mathbb{R}^N)$, we get $m_{r,\tau}^-:=\displaystyle \inf_{u\in \mathcal{M}_{r,\tau}^-}E(u)=\inf_{u\in \mathcal{M}_{r,\tau}^-}E(u)=m_{\tau}^-$, by symmetrization and the fact that $\mathcal{M}_{r,\tau}^-\subset\mathcal{M}_{\tau}^-$. Now, let us prove our final result:
			%	\begin{theorem}\label{Theorem 2}
				%		Let $N\geq 3$, $2<p<2+\frac{4s}{N}$, $0<\tau<\min\{\tau_0,\tau_1\}$ and $\mu>0$ be sufficiently large, then $m_{\tau}^-$ is achieved by a radially symmetric function $u_{\tau}^-\in H^1(\mathbb{R}^N)$. Furthermore, $u_{\tau}^-$ solves \eqref{prob} corresponding to some $\lambda_{\tau}^-<0$.
				%	\end{theorem}
			%	\begin{proof}
				\begin{myproof}{Theorem}{\ref{Theorem 2}}
					Let $\{\bar{u}_n\}$ be the minimizing sequence for $E$ on $\mathcal{M}_{r,\tau}^-$, then by Ekeland variational principle, \cite[Theorem~1.1]{Ghoussoub}, we can find a sequence $\{u_n\}\in \mathcal{M}_{r,\tau}^-$ such that
					\begin{equation}\label{4.20}
						\left\{
						\begin{array}{cl}
							\left\| \bar{u}_n-u_n\right\|_{H^1(\mathbb{R}^N)}\rightarrow 0 & \text{ as } n\rightarrow\infty,\\
							E(u_n) \rightarrow m_{r,\tau}^- & \text{ as } n\rightarrow \infty,\\
							M(u_n)\rightarrow 0 & \text{ as } n\rightarrow \infty,\\
							E'|_{\mathcal{M}_{r,\tau}^-}(u_n)\rightarrow 0 & \text{ as } n\rightarrow \infty.
						\end{array}
						\right.
					\end{equation}
					Now, by \eqref{4.20} we have
					\begin{eqnarray}\label{4.21}
						m_{r,\tau}^- & = & \lim_{n\rightarrow \infty}E(u_n)=\lim_{n\rightarrow \infty}\left(E(u_n)-\frac{M(u_n)}{2}\right)\nonumber\\
						& = & \lim_{n\rightarrow \infty}\left(\frac{1}{p}\left(\frac{p\gamma_p}{2}-1\right)\left\| u_n \right\|_p^p+\frac{(1-s)}{2}[u_n]^2+\left(\frac{2^*_{\alpha}-1}{22^*_{\alpha}}\right)A(u_n)\right),
					\end{eqnarray}
					and, since $E(u_n)\leq m_{r,\tau}^-+1$, for large $n\in \mathbb{N}$, by Gagliardo-Nirenberg inequality \eqref{G_N_inequality}
					\begin{eqnarray*}
						\frac{(2^*_{\alpha}-1)}{22^*_{\alpha}}T(u_n)^2 & \leq &  \frac{(2^*_{\alpha}-1)}{22^*_{\alpha}}\left\| \nabla u_n \right\|_2^2+\frac{(2^*_{\alpha}-s)}{22^*_{\alpha}}[u_n]^2\\
						& = & E(u_n)-\frac{1}{22^*_{\alpha}}M(u_n)+\frac{1}{p}\left(1-\frac{p\gamma_p}{22^*_{\alpha}}\right) \left\| u_n \right\|_p^p\\
						& \leq & m_{r,\tau}^-+1+\frac{C_{N,p}(22^*_{\alpha}-p\gamma_p)}{p22^*_{\alpha}}\tau^{p(1-\gamma_p)}T(u_n)^{p\gamma_p},
					\end{eqnarray*}
					thus, $\{u_n\}$ is bounded and hence weakly convergent upto a subsequence in $H^1(\mathbb{R}^N)$. Denoting the weakly convergent subsequence as $\{u_n\}$ itself, let $u_0\in H_r(\mathbb{R}^N)$ be such that $u_n\rightharpoonup u_0$, weakly. Thanks to the compact embedding $H_r(\mathbb{R})\hookrightarrow L^q(\mathbb{R}^N)$, for all $q\in (2,2^*)$, we get $u_n\rightarrow u_0$ in $L^p(\mathbb{R}^N)$. Next, we claim that $u_0\neq 0$.\\
					Suppose $u_0=0$, then 
					$$0 = \lim_{n\rightarrow \infty}M(u_n) = \lim_{n\rightarrow \infty}\left(\left\| \nabla u_n \right\|_2^2+s[u_n]^2-A(u_n)\right),$$
					and hence $\displaystyle \lim_{n\rightarrow \infty}\left(\left\| \nabla u_n \right\|_2^2+s[u_n]^2\right)= \lim_{n\rightarrow \infty}A(u_n)$. Since $\{u_n\}$ is bounded in $H^1(\mathbb{R}^N)$, the sequence $\{\left\| \nabla u_n\right\|_2^2+s[u_n]^2\}$ is convergent upto a subsequence in $\mathbb{R}$.
					Now, let $$l=\displaystyle \lim_{n\rightarrow \infty}\left(\left\| \nabla u_n \right\|_2^2+s[u_n]^2\right)= \lim_{n\rightarrow \infty}A(u_n),$$
					then by \eqref{S_alpha}, we get $l(S_{\alpha}^{2^*_{\alpha}}-l^{2^*_{\alpha}-1})\leq 0$, thus, either $l=0$ or $l\geq S_{\alpha}^{\frac{2^*_{\alpha}}{2^*_{\alpha}-1}}$. For $l\geq S_{\alpha}^{\frac{2^*_{\alpha}}{2^*_{\alpha}-1}}$,  by \eqref{4.21} we get:
					\begin{eqnarray*}
						m_{\tau}^-=m_{r,\tau}^- & = & \lim_{n\rightarrow \infty}\left(\frac{1}{p}\left(\frac{p\gamma_p}{2}-1\right)\left\| u_n \right\|_p^p+\frac{(1-s)}{2}[u_n]^2+\left(\frac{2^*_{\alpha}-1}{22^*_{\alpha}}\right)A(u_n)\right)\\
						& \geq & \lim_{n\rightarrow \infty}\left(\frac{2^*_{\alpha}-1}{22^*_{\alpha}}\right)A(u_n)\geq \left(\frac{2^*_{\alpha}-1}{22^*_{\alpha}}\right)S_{\alpha}^{\frac{2^*_{\alpha}}{2^*_{\alpha}-1}}>m_{\tau}+\left(\frac{2^*_{\alpha}-1}{22^*_{\alpha}}\right)S_{\alpha}^{\frac{2^*_{\alpha}}{2^*_{\alpha}-1}},
					\end{eqnarray*}
					but this contradicts \autoref{Lemma 4.1}. Also, if $l=0$, we will end up with $m_{r,\tau}^-=0$, but since $0<m_{\tau}^-=m_{r,\tau}^-$, we get a contradiction. Therefore, $u_0\neq 0$. Now, define $v_n:=u_n-u_0$, clearly $v_n\rightharpoonup 0$ in $H^1(\mathbb{R}^N)$.\\
					Case 1: $\left\| v_n\right\|_{H^1(\mathbb{R}^N)}\rightarrow 0 $.\\
					In this case, we get strong convergence of $\{u_n\}$ in $H^1(\mathbb{R}^N)$, and hence $u_0\in \mathcal{M}_{r,\tau}^-$ with $E(u_0)=m_{\tau}^-$ and hence $E'_{\mathcal{M}_{\tau}}(u_0)=0$. Thus, by \autoref{Lemma 4.3}, $u_0$ solves \eqref{prob} for some $\lambda_0\in \mathbb{R}$, and since $M(u)=0$, we have:
					\begin{equation*}
						\lambda_0\tau^2  =  \left\| \nabla u_0 \right\|_2^2+[u_0]^2-\mu \left\| u_0 \right\|_p^p-A(u_0)= (1-s)[u_0]^2+\mu(\gamma_p-1)\left\| u_0 \right\|_p^p<0,
					\end{equation*}
					for sufficiently large $\mu>0$. Hence, taking $u_{\tau}^-=u_0$ and $\lambda_{\tau}^-=\lambda_0$, we are done.\\
					Case 2: $\displaystyle \lim_{n\rightarrow \infty}\left\| v_n \right\|_{H^1(\mathbb{R}^N)}\neq 0$, that is, $\left\| v_n \right\|_{H^1(\mathbb{R}^N)} \geq \tilde{C}>0$ for large $n\in \mathbb{N}$.\\
					Let $\left\| u_0 \right\|_2=r_0$, then by Fatou's lemma, we have $0<r_0\leq \tau$. Now, either $A(v_n)\rightarrow 0$ or there exists a constant $\bar{C}>0$ such that $A(v_n)\geq \bar{C}$ for large $n\in \mathbb{N}$. Let us analyse the two subcases separately:\\
					Subcase 1: $A(v_n)\rightarrow 0$ as $n\rightarrow \infty$.\\
					Since $u_0\in S(r_0)$, by \autoref{Lemma 2.3}, there exists $c_0>0$ such that $c_0\circledast u_0 \in \mathcal{M}_{r,r_0}^-$. Thus, by \cite[lemma~2.4]{Moroz2013groundstates}, compact embedding of $H_r(\mathbb{R}^N)$ in $L^p(\mathbb{R}^N)$, Fatou's lemma and \autoref{Lemma 2.3} we get
					\begin{eqnarray}\label{4.22}
						m_{\tau}^- & = & \lim_{n\rightarrow \infty}E(u_n) \geq \lim_{n\rightarrow \infty}E(c_0 \circledast u_n) \nonumber\\
						& = & \lim_{n\rightarrow \infty}\left(\frac{c_0^2\left\| \nabla u_n \right\|_2^2}{2}+\frac{c_0^{2s}[u_n]^2}{2}-\frac{\mu c_0^{p\gamma_p}\left\| u_n \right\|_p^p}{p}-\frac{c_0^{22^*_{\alpha}}A(u_n)}{22^*_{\alpha}}\right)\nonumber\\
						& \geq & \frac{c_0^2\left\| \nabla u_0 \right\|_2^2}{2}+\frac{c_0^{2s}[u_0]^2}{2}-\frac{\mu c_0^{p\gamma_p}\left\| u_0 \right\|_p^p}{p}-\frac{c_0^{22^*_{\alpha}}A(u_0)}{22^*_{\alpha}}= E(c_0 \circledast u_0)\geq m_{r_0}^-,
					\end{eqnarray}
					also, since $0<r_0\leq \tau$, for any $u\in \mathcal{M}_{r_0}^-$, by \autoref{Lemma 4.2} we can find $v\in \mathcal{M}_{\tau}^-$ such that $E(u)>E(v)\geq \displaystyle \inf_{u\in \mathcal{M}_{\tau}^-}E(u)$ and hence $m_{r_0}^-\geq m_{\tau}^-$. Therefore, $m_{\tau}^-=m_{r_0}^-$. Now, we claim that $r_0=\tau$ and hence $u_{\tau}^-=c_0\circledast u_0$ is the required solution to \eqref{prob} corresponding to some $\lambda_{\tau}^-$ with $\lambda_{\tau}^-<0$ for sufficiently large $\mu>0$ as done in case 1.\\
					Suppose if $0<r_0<\tau<\min\{\tau_0,\tau_1\}$, then by \autoref{Lemma 4.2}, there exists $\bar{v}\in \mathcal{M}_{\tau}^-$ such that $E(c_0\circledast u_0)>E(\bar{v})$, then by \eqref{4.22} we have
					$$m_{r_0}^-=E(c_0\circledast u_0)>E(\bar{v})\geq m_{\tau}^-,$$
					but since $m_{r_0}^-=m_{\tau}^-$, we get contradiction, thus $r_0=\tau$.\\
					Subcase 2: $A(v_n)\geq \bar{C}>0$ for large $n\in \mathbb{N}$.\\
					For every $n\in \mathbb{N}$, define $$s_n:=\left(\frac{\left\| \nabla v_n\right\|_2^2}{A(v_n)}\right)^{\frac{1}{2(2^*_{\alpha}-1)}},$$
					clearly, by boundedness of $\{\frac{1}{A(v_n)}\}$ and $\{u_n\}$ in $H^1(\mathbb{R}^N)$, $\{s_n\}$ is a bounded sequence in $\mathbb{R}$. Now, since $u_0\in S(r_0)$, by \autoref{Lemma 2.3} there exists $c_0>0$ such that $c_0\circledast u_0\in \mathcal{M}_{r_0}^-$. We claim that $s_n\geq c_0$ upto subsequence.\\
					Suppose $s_n<c_0$ for all $n\in \mathbb{N}$, defining 
					$$E_0(u):=\frac{\left\| \nabla u \right\|_2^2}{2}-\frac{A(u)}{22^*_{\alpha}},$$
					by \autoref{Lemma 2.3}, Brezis Lieb lemma and \cite[lemma~2.4]{Moroz2013groundstates} we get,
					\begin{eqnarray}\label{4.23}
						m_{\tau}^- & = & \lim_{n\rightarrow \infty}E(u_n) \geq \lim_{n\rightarrow \infty}E(s_n\circledast u_u) =\lim_{n\rightarrow \infty}\left(E(s_n\circledast u_0)+E(s_n \circledast v_n)\right)\nonumber\\
						& \geq & \lim_{n\rightarrow \infty}\left(E(s_n\circledast u_0)+E_0(s_n \circledast v_n)\right)\geq m_{r_0}^++ \lim_{n\rightarrow \infty}E_0(s_n\circledast v_n).
					\end{eqnarray}
					Now, by \eqref{S_alpha}
					$$E_0(s_n\circledast v_n) = \left(\frac{2^*_{\alpha}-1}{22^*_{\alpha}}\right)\left(\frac{\left\| \nabla v_n \right\|_2^2 }{A(v_n)^{\frac{1}{2^*_{\alpha}}}}\right)^{\frac{2^*_{\alpha}}{2^*_{\alpha}-1}}\geq \left(\frac{2^*_{\alpha}-1}{22^*_{\alpha}}\right)S_{\alpha}^{\frac{2^*_{\alpha}}{2^*_{\alpha}-1}},$$
					thus, by \autoref{Lemma 4.2} 
					$$m_{\tau}^-\geq m_{r_0}^++\left(\frac{2^*_{\alpha}-1}{22^*_{\alpha}}\right)S_{\alpha}^{\frac{2^*_{\alpha}}{2^*_{\alpha}-1}}\geq m_{\tau}^++\left(\frac{2^*_{\alpha}-1}{22^*_{\alpha}}\right)S_{\alpha}^{\frac{2^*_{\alpha}}{2^*_{\alpha}-1}}.$$
					But, this is a contradiction to \autoref{Lemma 4.1}. Thus, there exists a subsequence (denoted as $\{s_n\}$ itself), such that $s_n\geq c_0$ for all $n\in \mathbb{N}$. Now, again proceeding as in \eqref{4.23}
					\begin{equation*}
						m_{\tau}^-  =  \lim_{n\rightarrow \infty}E(u_n) \geq \lim_{n\rightarrow \infty}E(c_0\circledast u_n)\geq \lim_{n\rightarrow \infty}\left(E(c_0\circledast u_0)+E_0(c_0\circledast v_n)\right) \geq E(c_0\circledast u_0),
					\end{equation*}
					because, $c_0\leq s_n$, which implies that $$\frac{c_0^{22^*_{\alpha}}A(v_n)}{\left\| \nabla v_n \right\|_2^2}\leq c_0^2,$$ and hence
					$$E_0(c_0\circledast v_n)\geq \left(\frac{2^*_{\alpha}-1}{22^*_{\alpha}}\right)c_0^{22^*_{\alpha}}A(v_n)\geq0.$$
					Therefore, $E(c_0\circledast u_0) \leq m_{\tau}^-$. Also, since $c_0\circledast u_0\in \mathcal{M}_{r_0}^-$, by \autoref{Lemma 4.2}
					$$m_{\tau}^-\geq E(c_0\circledast u_0)\geq m_{r_0}^-\geq m_{\tau}^-.$$
					Hence $E(c_0\circledast u_0) = m_{\tau}^-$, thus taking $u_{\tau}^-=c_0\circledast u_0$ we get the required result.
				\end{myproof}
				
			\end{document}